\documentclass[english]{amsart}
\usepackage{}
\usepackage{bbm}
\usepackage{amsfonts}
\usepackage{amsthm}
\usepackage{amsmath}
\usepackage{amssymb}
\usepackage{mathrsfs}
\usepackage{amscd}
\usepackage{mathdots}
\usepackage[all]{xy}

\newcommand{\I}{\mathbb{I}}

\newcommand{\W}{\mathbf{W}}
\newcommand{\INT}{\mathbb{Z}}
\newcommand{\RAT}{\mathbb{Q}}

\newcommand{\GG}{\mathbb{G}}
\newcommand{\A}{\mathcal {A}}
\newcommand{\V}{\mathcal {V}}
\newcommand{\ES}{\mathscr{S}}

\newcommand{\dr}{\mathrm{dR}}
\newcommand{\cris}{\mathrm{cris}}
\newcommand{\sdr}{s_{\dr}}

\newcommand{\EH}{\mathcal {H}}
\newcommand{\G}{\mathcal {G}}
\newcommand{\R}{\mathcal {R}}
\newcommand{\M}{\mathcal {M}}
\newcommand{\bM}{\mathbb {M}}

\newcommand{\intG}{G_{\mathbb{Z}_p}}
\newcommand{\inttG}{G_{\mathbb{Z}_{(p)}}}
\newcommand{\intV}{V_{\mathbb{Z}_p}}

\newcommand{\IIsom}{\mathbf{Isom}}
\newcommand{\Isom}{\mathrm{Isom}}
\newcommand{\AAut}{\mathbf{Aut}}
\newcommand{\Aut}{\mathrm{Aut}}

\newcommand{\Hom}{\mathrm{Hom}}
\newcommand{\Sh}{\mathrm{Sh}}
\newcommand{\HH}{\mathrm{H}}
\newcommand{\Hdr}{\mathrm{H}^1_{\mathrm{dR}}}
\newcommand{\lin}{\mathrm{lin}}
\newcommand{\GL}{\mathrm{GL}}
\newcommand{\GSp}{\mathrm{GSp}}
\newcommand{\Dieu}{\mathbb{D}}
\newcommand{\cent}{\mathrm{Ct}}
\newcommand{\newt}{\mathrm{Nt}}
\newcommand{\sigmu}{\mu^\sigma}
\newcommand{\RRZ}{\mathcal{R}\mathcal{Z}}
\newcommand{\RZ}{\mathrm{RZ}}
\newcommand{\Nil}{\mathrm{Nil}}
\newcommand{\ANil}{\mathrm{ANil}}
\newcommand{\nil}{\mathrm{nil}}
\newcommand{\fsm}{\mathrm{fsm}}
\newcommand{\Spf}{\mathrm{Spf}}
\DeclareMathOperator{\Spec}{\mathrm{Spec}}
\DeclareMathOperator{\Fil}{\mathrm{Fil}}

\normalsize \relpenalty=10000 \binoppenalty=10000

\begin{document}
\newtheorem{theorem}[subsubsection]{Theorem}
\newtheorem{lemma}[subsubsection]{Lemma}
\newtheorem{proposition}[subsubsection]{Proposition}
\newtheorem{corollary}[subsubsection]{Corollary}
\theoremstyle{definition}
\newtheorem{definition}[subsubsection]{Definition}
\theoremstyle{definition}
\newtheorem{construction}[subsubsection]{Construction}
\theoremstyle{definition}
\newtheorem{notations}[subsubsection]{Notations}
\theoremstyle{definition}
\newtheorem{asp}[subsubsection]{Assumption}
\theoremstyle{definition}
\newtheorem{set}[subsubsection]{Setting}
\theoremstyle{remark}
\newtheorem{remark}[subsubsection]{Remark}
\theoremstyle{remark}
\newtheorem{example}[subsubsection]{Example}
\theoremstyle{plain}
\newtheorem{introth}[section]{Theorem}

\makeatletter
\newenvironment{subeqn}{\refstepcounter{subsubsection}
$$}{\leqno{\rm(\thesubsubsection)}$$\global\@ignoretrue}
\makeatother
\author{Chao Zhang}

\address{
Yau Mathematical Sciences Center, Tsinghua University, Beijing, 100084, China.}

\email{czhang@math.tsinghua.edu.cn}

\title[Level $m$ strata and foliations]{Stratifications and foliations for good reductions of Shimura varieties of Hodge type}

\maketitle \setcounter{section}{-1} \setcounter{tocdepth}{2}
\begin{abstract}
Level $m$-stratifications on PEL Shimura varieties are defined and studied by Wedhorn using BT-$m$s with PEL structure, and then by Vasiu for general Hodge type Shimura varieties using Shimura $F$-crystals. The theory of foliations is established by Oort for Siegel modular varieties, and by Mantovan for PEL Shimura varieties. It plays an important role in Hamacher's work to compute the dimension of Newton strata of PEL Shimura varieties.
 
We study level $m$ stratifications on good reductions at $p>2$ of Shimura varieties of Hodge type by constructing certain torsors together with equivariant morphisms, and relating them to truncated displays. We then use the results obtained to extend the theory of foliations to these reductions. As a consequence, combined with results of Nie and Zhu, we get a dimension formula for Newton strata.
\end{abstract}
\tableofcontents
\newpage
\section[Introduction]{Introduction}

The theory of foliations is established by Oort in \cite{foliation-Oort} for Siegel modular varieties, and by Mantovan in \cite{coho of PEL} for good reductions of PEL shimura varieties. The theory roughly says that a Newton stratum is an ``almost product'' of a central leaf with an isogeny leaf. The goal of this paper is to extend such a theory to good reductions at char $p\geq 3$ of Shimura varieties of Hodge type. As a consequence, thanks to results and ideas in \cite{geo of newt PEL}, \cite{Nie Sian}, \cite{VW} and \cite{zhu-aff gras in mixed char}, we get a dimension formula for Newton strata.

The first step is to understand level $m$-stratifications on these reductions. They are defined and studied by Wedhorn in \cite{dimw} for those of PEL type using BT-$m$s with PEL structure, and by Vasiu in \cite{levmstra} for general Hodge type Shimura varieties using Shimura $F$-crystals.

We introduce a new way to understand level $m$-stratifications. It relies on explicit constructions of some torsors on the reductions, and it is closely related to truncated displays in \cite{smooth truncated display}. Our construction uses group schemes defined in \cite{levmstra}, and is greatly influenced by ideas and methods there. We refer to section 2 and 3 for details. We also consider \emph{classical level} $m$-\emph{stratifications} defined by geometric isomorphism types of BT-$m$s. We summarize the main results as follows.
\begin{introth}
Let $\ES_0$ be the good reduction of a Shimura variety of Hodge type.

(1) Each level $m$ stratum $\ES_0^s$ is a smooth locally closed subscheme of $\ES_{0,\overline{\kappa}}$. The closure $\overline{\ES_0^s}$ is a union of level $m$ strata, and $\ES_0^s\hookrightarrow\overline{\ES_0^s}$ is an affine immersion.

(2) There is a quasi-finite fppf cover $T=\Spec A\rightarrow\ES_0^s$, such that the BT-$m$ with additional structure is constant.

(3) Each level $m$ stratum is open and closed in its classical level $m$ stratum.
\end{introth}

By \cite{crysboud}, for $m$ big enough, all level $m$ strata are central leaves. Using the above theorem, as well results in \cite{geo of newt PEL}, \cite{LRKisin}, \cite{Nie Sian}, \cite{foliation-Oort} and \cite{Manin problem}, we can prove the followings.
\begin{introth}
Let $\ES_0$ be as above. Then

(1) Each central leaf $\ES_0^c$ is a smooth locally closed subscheme of $\ES_{0,\overline{\kappa}}$. The closure $\overline{\ES_0^c}$ is a union of central leaves, and $\ES_0^c\hookrightarrow\overline{\ES_0^c}$ is an affine immersion.

(2) Each central leaf is open and closed in its central leaf, and it is also closed in its Newton stratum.

(3) Central leaves in a Newton stratum $\ES_0^b$ are of the same dimension $\langle2\rho,\nu_G(b)\rangle$.
\end{introth}

We define Igusa towers in 4.2, and give the almost product morphism in 4.3. The ideals and technics are from \cite{RZ for spin}, \cite{LRKisin}, \cite{coho of PEL}, and \cite{foliation-Oort}. With results in \cite{zhu-aff gras in mixed char}, we deduce that
\begin{introth}
The Newton stratum $\ES_0^b$ is of dimension $\langle\rho,\mu+\nu(b)\rangle-\frac{1}{2}\mathrm{def}(b).$
\end{introth}

\newpage
\setcounter{section}{0}
\section[Preliminaries]{Preliminaries}

\subsection[Dieudonn\'{e} crystals and Dieudonn\'{e} modules]{Dieudonn\'{e} crystals and Dieudonn\'{e} modules}\label{Dieudonne modules}

We will first recall results related to crystals of $p$-divisible groups that will be used in this paper. Our main references are \cite{crystali dieudonne}, \cite{crys dieu via rigid} and \cite{crys rep and F-crys}.

Let $\kappa$ be a perfect field of characteristic $p>0$ and $T$ be a $W(\kappa)$-scheme such that $p$ is locally nilpotent. We refer to \cite{crystali dieudonne} for the definition of crystalline cite and crystals. Let $(T/W(\kappa))_{\cris}$ be the crystalline site of $T$ over $W(\kappa)$. There is a contravariant functor $G\mapsto \mathbb{D}(G)$ from the category of $p$-divisible groups over $T$ to the category of crystals (of quasi-coherent sheaves) over $(T/W(\kappa))_{\cris}$. This functor is defined using the Lie algebra of the universal vector extension of the dual $p$-divisible group $G^\vee$. The formation of $\mathbb{D}$ is compatible with base change. In particular, if $p=0$ on $T$, then the absolute Frobenius $\sigma$ on $T$ induces the relative Frobenius on $G$, and hence a morphism of crystals $\sigma^*(\mathbb{D}(G))\rightarrow \mathbb{D}(G)$.

Suppose now that $T_0$ is a $\kappa$-scheme, and $G_0$ is a
$p$-divisible group over $T_0$. Let $T_0\rightarrow T$ be an
object of $(T_0/W(\kappa))_{\cris}$ on which $p$ is locally
nilpotent, and $G$ be a lifting of $G_0$ to $T$. By construction
of $\mathbb{D}$, we have an isomorphism $\mathbb{D}(G_0)(T)
\stackrel{\simeq}{\rightarrow}\mathbb{D}(G)(T)$. Moreover, the
$O_T$-module $\mathbb{D}(G)(T)$ sits in an exact sequence
$$0\rightarrow(\mathrm{Lie}G)^\vee\rightarrow\mathbb{D}(G)(T)\rightarrow \mathrm{Lie}G^\vee
\rightarrow 0.$$

\begin{definition}
Let $T$ be a $\kappa$-scheme. A Dieudonn\'{e} crystal over $T$ is triple $(\mathcal {E},\varphi,v)$ where

(1) $\mathcal {E}$ is a crystal of finite locally free modules over $(T/W(\kappa))_{\cris}$,

(2) $\varphi:\sigma^*\mathcal {E}\rightarrow \mathcal {E}$ and $v:\mathcal {E}\rightarrow\sigma^*\mathcal {E}$ are homomorphisms of $O_{T,\cris}$-modules such that $\varphi\circ v=p\cdot \mathrm{id}_{\mathcal {E}}$ and $v\circ \varphi=p\cdot \mathrm{id}_{\sigma^*\mathcal {E}}$.
\end{definition}
It is well known that $\mathbb{D}(G)$ is a Dieudonn\'{e} crystal, for a $p$-divisible group $G/T$.

Let $A_0$ be a formally smooth $\kappa$-algebra, by a \emph{lifting} of $A_0$, we mean a $p$-adically complete flat $W(\kappa)$-algebra $A$ such that $A\otimes \kappa\cong A_0$. Such a lifting always exists by \cite{crys dieu via rigid} Lemma 1.1.2 and Lemma 1.2.2, it is unique by \cite{crys dieu via rigid} Remark 1.2.3 (b). Moreover, $A$ is formally smooth over $W(\kappa)$ (with respect to the $p$-adic topology), and the Frobenius $\sigma:A_0\rightarrow A_0$ lifts to $A$ (but NOT necessarily unique).

If $A_0$ is formally finitely generated (see \ref{R-Z for Hodge} for the definition), then $A$ is necessarily regular. This is because $W(\kappa)\rightarrow A$ is flat with geometrically regular special fiber, while $A$ is a quotient of $W(\kappa)[[x_1,\cdots,x_r]]\{y_1,\cdots,y_s\}$, the $p$-adic completion of $W(\kappa)[[x_1,\cdots,x_r]][y_1,\cdots,y_s]$, and hence all maximal ideals contain $p$. If $A=W(\kappa)[[x_1,\cdots,x_n]]$, the homomorphism given by Frobenius on $W(\kappa)$ and $p$-th power on indeterminants is a lift of the Frobenius on Let $A_0:=A\otimes \kappa$.

\begin{definition}
Fixing the pair $(A,\sigma)$, a Dieudonn\'{e} module over $A_0$ is quadruple $(M,\varphi,v,\nabla)$ where

(1) $M$ is a locally free $A$-module,

(2) $\nabla:M\rightarrow M\otimes \Omega^{1\wedge}_{A/W(k)}$ is an integrable, topologically quasi-nilpotent connection,

(3) $\varphi:\sigma^*M\rightarrow M$ and $v:M\rightarrow\sigma^*M$ are horizontal homomorphisms of $A$-modules such that $\varphi\circ v=p\cdot \mathrm{id}_{M}$ and $v\circ \varphi=p\cdot \mathrm{id}_{\sigma^*M}$.
\end{definition}
This definition is a special case of \cite{crys dieu via rigid} Definition 2.3.4, as $\mathrm{ker}(A\rightarrow A_0)=(p)$ is equipped with the natural PD-structure, and $\widehat{D}$ there is our $A$. We refer to \cite{crys dieu via rigid} Remark 2.2.4 c) for the definition of a topologically quasi-nilpotent connection.

\begin{proposition}\label{crystal vs module with connec}
The category of Dieudonn\'{e} crystals over $\Spec A_0$ is equivalent to the category of Dieudonn\'{e} modules over $A_0$.
\end{proposition}
\begin{proof}
This is a direct consequence of \cite{crys dieu via rigid} Proposition 2.2.2 and Remark 2.2.4 b) and h).
\end{proof}

We will use some constructions that are needed for the proof.

\begin{construction}\label{construction of connection}
The Dieudonn\'{e} module attached to a Dieudonn\'{e} crystal $\mathcal {E}$ is as follows. The $A$-module $M$ is $\varprojlim_n\mathcal {E}(A_n)$. Let $A(2)$ be the $p$-adic completion of the PD-envelop of $A\widehat{\otimes} A$ with respect to $\mathrm{ker}(A\widehat{\otimes} A\rightarrow A_0)$. There are homomorphisms $i_1:A\rightarrow A(2)$, $a\rightarrow a\otimes 1$ and $i_2:A\rightarrow A(2)$, $a\rightarrow 1\otimes a$. The crystal property of $\mathcal {E}$ gives an isomorphism $\varepsilon:i_2^*M\rightarrow i_1^*M$. Let $\theta:M\rightarrow i_2^*M$ be $m\mapsto 1\otimes m$, then $\nabla=\theta-\mathrm{id}_M\otimes 1$, $m\mapsto \varepsilon (1\otimes m)-m\otimes 1\in M\otimes (K/K^{[2]})=M\otimes \Omega^{1\wedge}_{A/W(\kappa)}$. Here $K$ is the kernel of $A(2)\rightarrow A$. The maps $F$ and $V$ follows from \cite{crys dieu via rigid} Remark 2.2.4 h).
\end{construction}
\begin{construction}\label{frob over diff rings}
Notations as above. Let $A'$ be a $p$-adically complete and $p$-torsion free $W(\kappa)$-algebra, equipped with a list of Frobenius $\sigma'$. Let $\iota:A\rightarrow A'$ be a homomorphism of $W(\kappa)$-algebras. Then $\iota^*(\mathcal {E})(A')$ is again an $A'$-module with connection and Frobenius. The module is just $M\otimes A'$ and the connection is just $\iota^*\nabla$. Since $\iota\circ \sigma$ and $\sigma'\circ \iota$ has the same reduction modulo $p$, the crystal property of $\mathcal {E}$ induces a canonical isomorphism $\varepsilon':\sigma'^*\iota^*(M)\rightarrow \iota^*\sigma^*(M)$ as follows. Let $\varepsilon^{[2]}:\big(A(2)/K^{[2]}\big)\otimes M\rightarrow M\otimes \big(A(2)/K^{[2]}\big)$ be the reduction modulo $K^{[2]}$ of $\varepsilon$, which is the isomorphism given by $\nabla+\mathrm{id}_M\otimes 1$, then $\varepsilon'=(\sigma'\iota\cdot \iota\sigma)^*(\varepsilon^{[2]})$. The Frobenius on $M\otimes A'$ is given by
$$\sigma'^*(M_{A'})=\sigma'^*\iota^*(M)\stackrel{\varepsilon'}{\longrightarrow}\iota^*\sigma^*M\rightarrow \iota^*M=M_{A'}.$$

Explicitly, $\varepsilon'$ has the following description: let $a_1,\cdots,a_n\in A$ be such that $da_1,\cdots,da_n$ form a basis of $\Omega^{1\wedge}_{A/W(\kappa)}$. Note that they exist Zariski locally. Then
\begin{subeqn}\label{eqn expli connection}
\varepsilon'(m\otimes 1)=\sum_{\underline{i}}\nabla(\partial)^{\underline{i}}(m\otimes1)\otimes \frac{z^{\underline{i}}}{\underline{i}!}.
\end{subeqn}
Here $\underline{i}=(i_1,\cdots, i_n)$ is a multi-index,
$\nabla(\partial)^{\underline{i}}=\nabla(\frac{\partial}{\partial
a_1})^{i_1}\cdots \nabla(\frac{\partial}{\partial a_n})^{i_n}$,
$z^{\underline{i}}=z_1^{i_1}\cdots z_n^{i_n}$ where
$z_i=\sigma'\circ \iota(a_i)-\iota\circ \sigma(a_i)$. The map
$\varepsilon'$ is well defined and independent of choices of
$a_1,\cdots,a_n$, and hence we always have $\varepsilon'$ by
Zariski gluing.
\end{construction}

\subsection[Integral canonical models]{Integral canonical models}\label{intcanmod}

We recall Kisin's construction of good reductions of Shimura varieties of Hodge type.

Let $(G,X)$ be a Shimura datum of Hodge type. Let $p$ be a prime. Assume that $G_{\mathbb{Q}_p}$ is quasi-split and split over an unramified extension of $\mathbb{Q}_p$. Then $G_{\mathbb{Q}_p}$ extends to a reductive group scheme $\intG$ over $\mathbb{Z}_p$. Let $K_p=\intG(\mathbb{Z}_p)$, for any compact open subgroup $K^p\subseteq G(\mathbb{A}_f^p)$ that is small enough, Kisin proves in \cite{CIMK} that the Shimura variety $\Sh_{K_pK^p}(G,X)$ has an integral canonical model provided that $p>2$.

We recall the constructions in \cite{CIMK}. Let
$i:(G,X)\hookrightarrow (\mathrm{GSp}(V,\psi),X')$ be a symplectic
embedding. By \cite{CIMK} Lemma 2.3.1, there exists a
$\mathbb{Z}_p$-lattice $V_{\mathbb{Z}_p}\subseteq
V_{\mathbb{Q}_p}$, such that
$i_{\mathbb{Q}_p}:G_{\mathbb{Q}_p}\subseteq\mathrm{GL}(V_{\mathbb{Q}_p})$
extends uniquely to a closed embedding $\intG\hookrightarrow
\mathrm{GL}(V_{\mathbb{Z}_p}).$ So there is a $\mathbb{Z}$-lattice
$V_{\INT}\subseteq V$ such that $\inttG$, the Zariski closure of $G$ in
$\mathrm{GL}(V_{\mathbb{Z}_{(p)}})$, is reductive, as the base
change to $\mathbb{Z}_p$ of $\inttG$ is $\intG$. Moreover, we can
assume $V_{\INT}$ is such that $V_{\INT}^\vee\supseteq V_{\INT}$. Let $d=|V_{\INT}^\vee/V_{\INT}|$,
$g=\frac{1}{2}\mathrm{dim}(V)$, $K=K_pK^p$, $E$ be the reflex
field of $(G,X)$ and $v$ be a place of $E$ over $p$, then the
integral canonical model $\ES_K(G,X)$ of $\mathrm{Sh}_K(G,X)$ is
constructed as follows. We can choose
$K'\subseteq\mathrm{GSp}(V,\psi)(\mathbb{A}_f)$ small enough such
that $K'\supseteq K$, that
$\mathrm{Sh}_{K'}(\mathrm{GSp}(V,\psi),X)$ affords a moduli
interpretation, and that the natural morphism
$$f:\mathrm{Sh}_K(G,X)\rightarrow
\mathrm{Sh}_{K'}(\mathrm{GSp}(V,\psi),X)_E$$ is a closed embedding.
Let $\mathscr{A}_{g,d,K'/\mathbb{Z}_{(p)}}$ be the moduli scheme
of abelian schemes over $\mathbb{Z}_{(p)}$-schemes with a
polarization of degree $d$ and level $K'$ structure. Then
$\mathscr{A}_{g,d,K'/\mathbb{Z}_{(p)}}\otimes
\mathbb{Q}=\mathrm{Sh}_{K'}(\mathrm{GSp}(V,\psi),X)$. The integral
canonical model $\ES_K(G,X)$ is the normalization of the Zariski
closure of $\mathrm{Sh}_K(G,X)$ in
$\mathscr{A}_{g,d,K'/\mathbb{Z}_{(p)}}\otimes O_{E,(v)}$.

\begin{theorem}
The $O_{E,(v)}$-scheme $\ES_K(G,X)$ is smooth, and morphisms in the inverse system $\varprojlim_{K^p}\ES_K(G,X)$ are \'{e}tale.
\end{theorem}
\begin{proof}
This is \cite{CIMK} Theorem 2.3.8.
\end{proof}
\begin{remark}
The morphism $\ES_K(G,X)\rightarrow \mathscr{A}_{g,d,K'/O_{E,(v)}}$ is finite, as $\mathscr{A}_{g,d,K'/O_{E,(v)}}$ is Nagata.
\end{remark}
Moreover, the scheme $\ES_K(G,X)$ is uniquely determined by the Shimura datum and the group $K$ in the sense that $\varprojlim_{K^p}\ES_K(G,X)$ satisfies a certain extension property (see \cite{CIMK} 2.3.7 for a precise statement) and there is an $G(\mathbb{A}_f^p)$-action on $\varprojlim_{K^p}\ES_K(G,X)$ extending the one on $\varprojlim_{K^p}\Sh_K(G,X)$.

Let $\A\rightarrow \ES_K(G,X)$ be the pull back to $\ES_K(G,X)$ of the universal abelian scheme on $\mathscr{A}_{g,d,K'/\mathbb{Z}_{(p)}}$, and $\V$ be $\Hdr(\A/\ES_K(G,K))$. Kisin also constructed in \cite{CIMK} certain sections of $\V^\otimes$ which will play an important role in this paper. By \cite{CIMK} Proposition 1.3.2, there is a tensor $s\in V_{\INT_{(p)}}^\otimes$ defining $\inttG\subseteq \GL(V_{\INT_{(p)}})$. This tensor gives a section $s_{\dr/E}$ of $\V_{\Sh_K(G,X)}^\otimes$, which is actually defined over $O_{E,(v)}$. More precisely, we have the following result.

\begin{proposition}
The section $s_{\dr/E}$ of $\V_{\Sh_K(G,X)}^\otimes$ extends to a section $s_{\dr}$ of~$\V^\otimes$.
\end{proposition}
\begin{proof}
This is \cite{CIMK} Corollary 2.3.9.
\end{proof}

The extended section $s_{\dr}$ has a lot of good properties. One sees easily that it is a locally direct summand of $\V^\otimes$. Let $x$ be a closed point in the special fiber of $\ES_K(G,X)$, $\widetilde{x}$ be a lifting to a $W(k(x))$-point of $x$, then $s_{\dr,\widetilde{x}}\in\V^\otimes_{\widetilde{x}}$ is invariant under the Frobenius on $\V^\otimes_{\widetilde{x}}$ (see \cite{CIMK} Corollary 1.4.3). Here we use the canonical isomorphism $\V_{\widetilde{x}}\cong \HH^1_{\cris}(\A_x/k(x))$. Let $\widehat{R}$ be the completion of the stalk at $x$ of $\ES_K(G,X)$ with respect to the maximal ideal, and $\Dieu$ be the contravariant Dieudonn\'{e} functor, then $\sdr$ is parallel with respect to the connection on $\Dieu(\A_{\widehat{R}})(\widehat{R})$.

\subsection[Definitions of some stratifications]{Definitions of some stratifications}

We define various stratifications on reductions of Shimura varieties of Hodge type and varieties over them. Note that it is technically necessary to define stratifications on varieties over reductions of Shimura varieties of Hodge type.

Let $\ES_0$ be the special fiber of $\ES_K(G,X)$. It is smooth over $\kappa=O_{E,(v)}/(v)$. Using the morphism $\ES_0\rightarrow \mathscr{A}_{g,d,K'/\kappa}$, one gets \emph{classical Newton stratification} (resp. \emph{classical level $m$ stratification}, resp. \emph{classical central leaves}) on an $\ES_0$-scheme $X$ by putting together points of $X_{\overline{\kappa}}$ whose attached $p$-divisible groups (resp. BT-$m$s, resp. $p$-divisible groups) are geometrically isogenous (resp. geometrically isomorphic, resp. geometrically isomorphic). Let $i$ be a geometrically isogenous (resp. isomorphism, resp. isomorphism) class of $p$-divisible groups (resp. BT-$m$s, resp. $p$-divisible groups), we will write $X^{\mathrm{cN},i}$ (resp. $X^{\mathrm{cL},i}$, resp. $X^{\mathrm{cC},i}$) for the corresponding stratum. The following statement should be well known.
\begin{proposition}
Notations as above, we have

(1) Each $X^{\mathrm{cN},i}$ (resp. $X^{\mathrm{cL},i}$, resp. $X^{\mathrm{cC},i}$) is locally closed in $X_{\overline{\kappa}}$.

(2) Each classical central leaf is closed in the classical Newton stratum containing it.
\end{proposition}
\begin{proof}
The proof of \cite{dimw} Proposition 1.8 implies that the stack of BT-$m$s (of given height) is an algebraic stack. So $X^{\mathrm{cL},i}$ is locally closed. The locally closedness for $X^{\mathrm{cN},i}$ is in \cite{foliation-Oort} 2.1. Statement (2) is \cite{foliation-Oort} Theorem 2.2. This implies that $X^{\mathrm{cC},i}$ is locally closed.
\end{proof}
We are interested in refinements of the above stratifications. Let $x\in X$ be a point (not necessarily closed) and $\overline{x}$ be the geometric point of $X$ induced by embedding $k(x)$ to an algebraically closed field $k(\overline{x})$. Then we have a Frobenius invariant tensor $s_{\cris,\overline{x}}\in \V^\otimes_{W(k(\overline{x}))}=\HH^1_{\cris}(\A_{\overline{x}},k(\overline{x}))^\otimes$. Let $W_m$ be the ring of Witt vectors of length $m$, then by passing to $W_m(k(\overline{x}))$, we have a tensor $s_{\cris,\overline{x}}\in \V^\otimes_{W_m(k(\overline{x}))}=\Dieu(\A_{\overline{x}}[p^m])^\otimes$.

\begin{definition}
The Newton stratification on $X$ is a decomposition of the topological space $X=\coprod_{c\in C} X^{\mathrm{N},c}$ such that $x,y\in X$ are in the same subset $X^{\mathrm{N},c}$ if and only if there is an algebraically closed field $k$ with embeddings $k(x)\hookrightarrow k$ and $k(y)\hookrightarrow k$ and induced geometric points $\overline{x}$ and $\overline{y}$, such that there is an isomorphism $\HH^1_{\cris}(\A_{\overline{x}},k(\overline{x}))\otimes B(k)\rightarrow\HH^1_{\cris}(\A_{\overline{y}},k(\overline{y}))\otimes B(k)$ mapping $s_{\cris,\overline{x}}$ to $s_{\cris,\overline{y}}$. Each $X^{\mathrm{N},c}$ is called a Newton stratum.
\end{definition}

We can also define level $m$-stratifications.

\begin{definition}
The level $m$ stratification on $X$ is a decomposition of the topological space $X=\coprod_{j\in J} X^{\mathrm{L},j}$ such that $x,y\in X$ are in the same subset $X^{\mathrm{L},j}$ if and only if there is an algebraically closed field $k$ with embeddings $k(x)\hookrightarrow k$ and $k(y)\hookrightarrow k$ and induced geometric points $\overline{x}$ and $\overline{y}$, such that there is an isomorphism $\Dieu(\A_{\overline{x}}[p^m])\rightarrow\Dieu(\A_{\overline{y}}[p^m])$ mapping $s_{\cris,\overline{x}}$ to $s_{\cris,\overline{y}}$. Each $X^{\mathrm{L},j}$ is called a level $m$ stratum.
\end{definition}
\begin{definition}
Let $x\in \ES_0(\overline{\kappa})$ be a point. The central leaf $C_x$ of $\ES_0$ crossing $x$ is the subset $y\in \ES_{0,\overline{\kappa}}$ such that there exists an algebraically closed field $k$ with embeddings $k(x)\hookrightarrow k$ and $k(y)\hookrightarrow k$ and induced geometric points $\overline{x}$ and $\overline{y}$, and an isomorphism $\HH^1_{\cris}(\A_{\overline{x}},k(\overline{x}))\rightarrow\HH^1_{\cris}(\A_{\overline{y}},k(\overline{y}))$ of Dieudonn\'{e} modules mapping $s_{\cris,\overline{x}}$ to $s_{\cris,\overline{y}}$.
\end{definition}

For simplicity, we will skip the superscript $\mathrm{cN}$, $\mathrm{cL}$, $\mathrm{cC}$, $\mathrm{N}$ and $\mathrm{L}$ when there is no risk of confusion.

The level $1$ stratification is precisely the Ekedahl-Oort stratification defined and studied in \cite{EOZ}. Each classical Newton stratum (resp. classical level $m$ stratum, resp. classical central leaf) is a union of Newton strata (resp. level $m$ strata, resp. central leaves). Moreover, we have the following result of Vasiu.
\begin{proposition}
Each Newton stratum is a union of connected components of the classical Newton stratum containing it. In particular, Newton strata are locally closed in $X_{\overline{\kappa}}$.
\end{proposition}
\begin{proof}
This is \cite{Manin problem} Theorem 5.3.1 (b).
\end{proof}

\subsection[Dilatations and some group theoretic settings]{Dilatations and some group theoretic settings}

The Shimura datum $(G,X)$ determine a cocharacter $\mu:\GG_{m,W(\kappa)}\rightarrow G_{\INT_p}\otimes W(\kappa)$ which is unique up to $G_{\INT_p}(W(\kappa))$-conjugacy. We introduce in this subsection some group theoretic settings which are essentially from \cite{levmstra} section 4. For simplicity, we write $G_R$ for $G_{\INT_p}\otimes R$, for a $\INT_p$-algebra $R$. Let $P_+\subseteq G_{W(\kappa)}$ (resp. $L\subseteq G_{W(\kappa)},\ P_-\subseteq G_{W(\kappa)}$) be the subgroup whose Lie algebra is the submodule of $\mathrm{Lie}(G_{W(\kappa)})$ of non-negative weights (resp. of weight 0, of non-positive weights) with respect to $\mu$ composed with the adjoint action of $G_{W(\kappa)}$ on $\mathrm{Lie}(G_{W(\kappa)})$. Then $P_+$ and $P_-$ are parabolic subgroups of $G_{W(\kappa)}$ in opposite position, and $L$ is the common Levi subgroup of $P_+$ and $P_-$.

To construct what we need, we need to introduce dilatations. Let $R$ be a D.V.R. with uniformizer $t$ and residue field $k$, $X$ be a smooth scheme over $R$ and $Y_k\subseteq X_k$ be a closed subscheme which is smooth over $k$. Let $I$ be the ideal defining $Y_k\subseteq X$, $\widetilde{X}$ be the blow up of $Y_k$ on $X$ and $X'\subseteq \widetilde{X}$ be the open subscheme such that $IO_{\widetilde{X}}$ is generated by $t$. Following \cite{Neron Model}, $X'$ is called the dilatation of $Y_k$ on $X$.
\begin{proposition}\label{dilatation}
\

(1) The $R$-scheme $X'$ is smooth.

(2) The natural $R$-morphism $u:X'\rightarrow X$ whose generic fiber is an isomorphism is universal in the following sense.

For for any flat $R$-scheme $Z$ and any $R$-morphism $v:Z\rightarrow X$ such that $v_k$ factors through $Y_k$, there is a unique $R$-morphism $v':Z\rightarrow X'$ satisfying $v=u\circ v'$.

(3) Dilatations commute with products: Let $X_i, i=1,2$ be smooth $R$-schemes, and $Y_i\subseteq X_i\otimes k$ be closed smooth subvarieties, then $(X_1\times_R X_2)'$, the dilatation of $Y_1\times_kY_2$ on $X_1\times_R X_2$, is canonically isomorphic to $X'_1\times_R X'_2$. In particular, if $X$ is an $R$-group scheme and $Y_k\subseteq X_k$ is a subgroup scheme, then $X'$ is a group scheme over $R$, and the natural morphism $u:X'\rightarrow X$ is a group homomorphism.
\end{proposition}
\begin{proof}
The first statement is \cite{Neron Model}  3.2 Proposition 3, the second statement is \cite{Neron Model} 3.2 Proposition 1, and the last statement is \cite{Neron Model} 3.2 Proposition 2 (d).
\end{proof}

Let $P_{+,0}$ be the special fiber of $P_+$, and $G'$ be the dilatation of $P_{+,0}$ on $G_{W(\kappa)}$ with natural homomorphism $u:G'\rightarrow G$. Let $U_+$ (resp. $U_-$) be the unipotent subgroup of $P_+$ (resp. $P_-$), and $H$ be $U_+\times L\times U_-$. Let $f:H\rightarrow G_{W(\kappa)}$ be the natural morphism induced by product, and $f_-:H\rightarrow H$ be given by $(a,b,c)\rightarrow (a,b,c^p)$. Let $f_{0-}:H\rightarrow G_{W(\kappa)}$ be the composition $f\circ f_-$. Then $f_{0-}$ is generically an open embedding. Note that on the special fiber, $f_{0-}$ factors through $P_{+,0}$, so by Proposition \ref{dilatation} (2), $f_{0-}$ factors through $u:G'\rightarrow G$. The morphism $H\rightarrow G'$ is denoted by~$u'$.

Note that $u'$ is also generically an open embedding, so we have sequences of injections
$$H(W(\overline{\kappa}))\hookrightarrow G'(W(\overline{\kappa}))\hookrightarrow G_{W(\kappa)}(W(\overline{\kappa}))\ \ \ \ \ \text{and}$$
$$H(W(\overline{\kappa})[\varepsilon]/(\varepsilon^2))\hookrightarrow G'(W(\overline{\kappa})[\varepsilon]/(\varepsilon^2))\hookrightarrow G_{W(\kappa)}(W(\overline{\kappa})[\varepsilon]/(\varepsilon^2)).$$
By Proposition \ref{dilatation} (2),
$G'(W(\overline{\kappa}))=\{g\in
G_{W(\kappa)}(W(\overline{\kappa}))\mid g\ \mathrm{mod}\ p\in
P_{+,0}(\overline{\kappa})\}$ and
$G'(W(\overline{\kappa})[\varepsilon]/(\varepsilon^2))=\{g\in
G_{W(\kappa)}(W(\overline{\kappa})[\varepsilon]/(\varepsilon^2))\mid
g\ \mathrm{mod}\ p\in
P_{+,0}(\overline{\kappa}[\varepsilon]/(\varepsilon^2))\}$. One
sees easily that the image of $H(W(\overline{\kappa}))$ (resp.
$H(W(\overline{\kappa})[\varepsilon]/(\varepsilon^2))$) in
$G_{W(\kappa)}(W(\overline{\kappa}))$ (resp.
$G_{W(\kappa)}(W(\overline{\kappa})[\varepsilon]/(\varepsilon^2))$)
equals to $G'(W(\overline{\kappa}))$ (resp.
$G'(W(\overline{\kappa})[\varepsilon]/(\varepsilon^2))$). So the
natural morphism of smooth affine $\kappa$-schemes (resp.
$W_m(\kappa)$-schemes) $u'_1:H_\kappa\rightarrow G'_\kappa$ (resp.
$u'_m:H_{W_m(\kappa)}\rightarrow G'_{W_m(\kappa)}$) is an
isomorphism.

From now on, we will identify $H_\kappa$ (resp. $H_{W_m(\kappa)}$) with $G'_\kappa$ (resp. $G'_{W_m(\kappa)}$) via $u'_1$ (resp. $u'_m$), and view $H_\kappa$ (resp. $H_{W_m(\kappa)}$) as a group scheme via the identification. We will write $H_m$ (resp. $G'_m$, $U_{+,m}$, $U_{-,m}$, $L_m$, $G_m$) for $H_{W_m(\kappa)}$ (resp. $G'_{W_m(\kappa)}$, $U_{+,W_m(\kappa)}$, $U_{-,W_m(\kappa)}$, $L_{W_m(\kappa)}$, $G_{W_m(\kappa)}$). We need a precise description of the group structure on $H_m$. We start with $u':H\rightarrow G'$. To describe $u'$, it suffices to consider the map $H(H)\rightarrow G'(H)$, but $H(H)\subseteq H_{\RAT}(H_{\RAT})$ as well as $G'(H)\subseteq G'_{\RAT}(H_{\RAT})$, so for $(h_1,h_2,h_3)\in H(H)$, $u'((h_1,h_2,h_3))=h_1h_2h_3^p$. This means that For $(h_1,h_2,h_3)$ and $(g_1,g_2,g_3)$ in $H(H)$, the product $u'((h_1,h_2,h_3))\cdot u'((g_1,g_2,g_3))$ is $h_1h_2h_3^pg_1g_2g_3^p$. When we reduce $u'$ by $p^m$, we get an isomorphism, and the group structure on $H_{m}$ is such that $(h_1,h_2,h_3)\cdot(g_1,g_2,g_3)=(f_1,f_2,f_3)$, where $f_1f_2f_3=h_1h_2h_3^pg_1g_2g_3^p$.

There is a homomorphism $p_m:G'_{m}\rightarrow P_{+,m}$ given by $(h_1,h_2,h_3)\mapsto (h_1,h_2)$. It is direct that $p_m\circ u'_m|_{P_{+,m}}$ is the identity. We remark that the construction of $p_m$ relies on the fact that $u'_m$ is an isomorphism, and such a homomorphism does NOT exist over $W(\kappa)$.

\subsubsection{The functor $\W_m$}\label{subsection greenberg}

Let $X/W(\kappa)$ be of finite type. Let $\W_m(X)$ be the presheaf on the category of affine $\kappa$-schemes such that $\W_m(X)(R)=X(W_m(R))$, for all $\kappa$-algebra~$R$. Clearly, a morphism $f:X\rightarrow Y$ of $W(\kappa)$-schemes induces a natural transformation $\W(f):\W_m(X)\rightarrow \W_m(Y)$. The functor $\W_m(-)$ has the following properties.

\begin{theorem}\label{greenberg functor}\

(1) The functor $\W_m(X)$ is represented by a $\kappa$-scheme, $\W_m(-)$ is compatible with fiber products.

(2) If $X$ is affine (resp. separated), then $\W_m(X)$ is also affine (resp. separated).

(3) If $f:X\rightarrow Y$ is an immersion (resp. open immersion, resp. closed immersion), then $\W(f):\W_m(X)\rightarrow \W_m(Y)$ is also an immersion (resp. open immersion, resp. closed immersion). Moreover, if $\{X_i\}_{i\in I}$ is an open covering of $X$, then $\{\W(X_i)\}_{i\in I}$ is an open covering of $\W_m(X)$.

(4) Let $X_0$ be the special fiber of $X$, and $F:X_0\rightarrow X_0$ be the absolute Frobenius. If $X/W(\kappa)$ is smooth, then the natural morphism $\W_{m+1}(X)\rightarrow \W_m(X)$ is a $F^{m*}\Omega^{1\vee}_{X_0/\kappa}$-torsor. In particular, $\W_n(X)$ is smooth.
\end{theorem}
\begin{proof}
The first three statements are in \cite{greenberg1} page 643, and the last statement follows from \cite{greenberg2} Proposition 2 by using Case 2 on page 263.
\end{proof}
\begin{remark}
For $m=1$, we have $\W_1(X)=X_\kappa$. If $X/W(\kappa)$ is smooth (resp. a group scheme), then $\W_m(X)$ is of dimension $m\cdot\mathrm{dim}(X_\kappa)$ (resp. a group scheme over $\kappa$).
\end{remark}
\begin{remark}
If $f:X\rightarrow S$ is a smooth morphism of $W(\kappa)$-schemes that are of finite type, then the induced morphism $\W_m(X)\rightarrow \W_m(S)$ is also smooth, as one can check that nilpotent immersions lifts locally.
\end{remark}
Now we come back to notations introduced before 1.4.2. Let $H_m$ be $\W_m(H)$ which is the same as $\W_m(G')$ and $G_m$ be $\W_m(G_{W(\kappa)})$. Let $\sigma:H_m\rightarrow H_m$ be the morphism which maps $h\in H_m(R)=H(W_m(R))$ to $W_m(R)\stackrel{\text{Frob}}{\longrightarrow}W_m(R)\stackrel{h}{\rightarrow}G'$, for any $\kappa$-algebra $R$. Then there is an action $T_m:H_m\times_\kappa G_m\rightarrow G_m$ given by $(h,g)\rightarrow hg\sigma(h)^{-1}$, where the multiplication on the right-hand-side is induced by the composition $\xymatrix{H_{W_m(\kappa)}\ar[r]^{u'_m}_{\cong} &G'_{W_m(\kappa)}\ar[r]^u &G_{W_m(\kappa)}}.$ We remark that our notations here conflicts with previous notations where $H_m$ is defined to be the reduction mod $p^m$ of $H$. BUT we INTEND to do this, as these two schemes are related by the functor $\W_m$, and one should be viewed as the realization of the other in the different world.

By the smoothness of $H_m$ and $G_m$, we have the following corollary.

\begin{corollary}
The quotient $[H_m\backslash G_m]$ induced by $T_m$ is a smooth Artin stack over $\kappa$.
\end{corollary}

For $m'\geq m$, we have natural morphisms $H_{m'}\rightarrow H_m$ and $G_{m'}\rightarrow G_m$ inducing a commutative diagram
$$\xymatrix{
H_{m'}\times G_{m'}\ar[r]^(0.65){T_{m'}}\ar[d]&G_{m'}\ar[d]\\
H_m\times G_m\ar[r]^(0.65){T_m}& G_m.
}$$
Hence we have a natural morphism of algebraic stacks $[H_{m'}\backslash G_{m'}]\rightarrow[H_m\backslash G_m]$.
\begin{theorem}\label{cent are big level m}
There exists an integer $N>0$, such that for any $m'\geq m\geq N$, the natural morphism $[H_{m'}\backslash G_{m'}]\rightarrow[H_m\backslash G_m]$ induces a homeomorphism on topological spaces.
\end{theorem}
\begin{proof}
This is Vasiu's \cite{crysboud} MAIN THEOREM A.
\end{proof}
\newpage

\section[Level $m$ stratifications]{Level $m$ stratifications}

\subsection[Constructing torsors over $\ES$ and $\ES_0$]{Constructing torsors over $\ES$ and $\ES_0$}
We expect a smooth morphism $\ES_0\rightarrow [H_m\backslash G_m]$ whose fibers are level $m$ strata. But we can only do this over a Zariski cover. Fortunately, this is enough for the study of level $m$ stratifications. The construction here is a generalization of our construction in \cite{EOZ}. For simplicity, we will write $\ES$ for $\ES_K(G,X)\otimes W(\kappa)$.

We use notations as in 1.1. Let $\I=\IIsom\big((V_{\INT_p}^\vee,s)\otimes O_{\ES}, (\V,s_{\dr})\otimes O_{\ES}\big)$, it is a $G_{W(\kappa)}$-torsor over $\ES$ (see e.g. \cite{EOZ}). Let $L^1\subseteq L_{\INT_p}^\vee$ be the submodule of weight 1 with respect to $\mu$, let $\V^1\subseteq \V$ be the Hodge filtration. Let $\I_+\subseteq \I$ be the closed subscheme given by $\I_+=\IIsom\big((V_{\INT_p}^\vee\supseteq L^1,s)\otimes O_{\ES}, (\V\supseteq \V^1,s_{\dr})\otimes O_{\ES}\big)$. Then $\I_+$ is a $P_+$-torsor over $\ES$ (see \cite{EOZ}).


Let $\I_{+,0}$ be the special fiber of $\I_+$ viewed as a $W(\kappa)$-scheme. Then $\I_{+,0}$ is smooth over $\ES_0$ and hence smooth over $\kappa$. Let $\I'$ be the dilatation of $\I_{+,0}$ on $\I$. Then by Proposition \ref{dilatation}, $\I'$ is smooth over $W(\kappa)$.
\begin{proposition}\label{dilatation is a torsor}
The $W(\kappa)$-scheme $\I'$ is naturally a $G'$-torsor over $\ES$. Moreover, we have a commutative diagram of $\ES$-morphisms
$$\xymatrix{
G'\times \I'\ar[r]\ar[d]^{u_1\times u}&\I'\ar[d]^u\\
G_{W(\kappa)}\times \I\ar[r]^(0.65)\rho& \I,
}$$
where the horizontal morphisms are actions, and the vertical ones are those induced by dilatations.
\end{proposition}
\begin{proof}
The natural composition $\I'\stackrel{u}{\rightarrow} \I\stackrel{\pi}{\rightarrow} \ES$ makes $\I'$ an $\ES$-scheme. We first prove that $\I'$ is faithfully flat over $\ES$.

We claim that $\I'$ is smooth over $\ES$. We only need to check this at points on $\ES_0$, as over $\Sh_K(G,X)$, the morphism $\I'\rightarrow \I$ induces a canonical isomorphism $\I'|_{\Sh_K(G,X)}\rightarrow \I|_{\Sh_K(G,X)}$. Let $\mathscr{I}$ be the ideal sheaf of $O_{\I}$ defining $\I_+$, then $O_{\I'}=O_{\I}[\frac{\mathscr{I}}{p}]$. Let $x'\in \I'|_{\ES_0}$ be a point, and $x=\pi\circ u(x')$ be its image in $\ES_0$, we will find open affine neighbourhoods $x'\in U'\subseteq \I'$ and $x\in V\subseteq \ES$, such that $\pi\circ u(U')\subseteq V$ and that $U'\rightarrow V$ is smooth. Still write $x'$ for its image in $\I$, then by the $\ES$-smoothness of $\I$ and $\I_+$, there is an open affine neighbourhood $U\subseteq \I$ of $x'$ and an open affine neighbourhood $V\subseteq \ES$ of $x$, such that

1. there are elements $a_1,\cdots, a_i\in \mathscr{I}_U$ such that $(a_1,\cdots, a_i)=\mathscr{I}_U$

2. there are elements $b_1,\cdots, b_t\in O_U$ such that $da_1,\cdots, da_i,db_1,\cdots,db_t$ form a basis of $\Omega^1_{O_U/O_V}$.\\
The homomorphism $R[x_1,\cdots,x_i,y_1,\cdots,y_t]\rightarrow O_U$ given by $x_j\rightarrow a_j$, $ 1\leq j\leq i$ and $y_s\rightarrow b_s$, $1\leq s\leq t$ induces a decomposition $U\stackrel{v}{\rightarrow}\mathbb{A}_V^{i+t}\rightarrow V$ with $v$ \'{e}tale. Now take $U'\subseteq \I'$ to be $u^{-1}(U)$, then $O_{U'}=O_U[\frac{\mathscr{I}_U}{p}]=O_U\otimes_{O_{\mathbb{A}_V^{i+t}}}
O_{\mathbb{A}_V^{i+t}}[\frac{x_1}{p},\cdots,\frac{x_i}{p}]$ is smooth over $O_{\mathbb{A}_V^{i+t}}[\frac{x_1}{p},\cdots,\frac{x_i}{p}]$, which is a polynomial algebra over $O_V$, and hence $U'$ is smooth over $V$.

Now we prove that $\I'\rightarrow \ES$ is surjective on topological spaces.

Let $x$ be a geometric point of $\ES_0$ with residue field $k=\overline{k}$, and $\widetilde{x}\in \ES(W(k))$ be a lifting of $x$. The $\ES$-scheme $\I_+$ is flat, so $\widetilde{x}$ induces an exact sequence $0\rightarrow \mathscr{I}\otimes_{O_\ES}W(k)\rightarrow O_{\I_{\widetilde{x}}}\rightarrow O_{\I_{+,\widetilde{x}}}\rightarrow 0.$ Let $(\I_{\widetilde{x}})'$ be the dilatation of the special fiber of $\I_{+,\widetilde{x}}$ on~$\I_{\widetilde{x}}$, then $O_{(\I_{\widetilde{x}})'}=O_{\I_{\widetilde{x}}}[\frac{\mathscr{I}\otimes_{O_\ES}W(k)}{p}]$. Note that there is a natural surjection $O_{\I'}\otimes_{O_\ES} W(k)\rightarrow O_{(\I_{\widetilde{x}})'}$, so $\I'_x$ contains the special fiber of $(\I_{\widetilde{x}})'$ which is non-empty. Note that the closed embedding $(\I_{\widetilde{x}})'\hookrightarrow \I'_{\widetilde{x}}$ is actually an isomorphism, as they are both smooth affine over $W(k)$ of the same dimension.

To prove that $\I'$ is a $G'$-torsor over $\ES$, we still need to construct an action $\rho':G'_\ES\times_\ES \I'\rightarrow \I'$ and show that the morphism $\xymatrix@C=1.2cm{G'_\ES\times_\ES \I'\ar[r]^(0.55){(\rho', pr_2)} &\I'\times_\ES \I'}$ is an isomorphism.

To construct $\rho'$, consider the diagram of $\ES$-morphism
$$\xymatrix@C=0.2cm{
G'_\ES\times_\ES\I'\ar@{=}[r]&G'\times_{W(\kappa)}\I'\ar[d]^{u_1\times u} &&& \I'\ar[d]^{u}\\
& G_{W(\kappa)\times_{W(\kappa)}}\I'\ar[rrr]^(0.7){\rho}&&& \I
}$$
By Proposition \ref{dilatation} 3, the morphism $u_1\times u$ is the dilatation of $P_{+,0}\times_\kappa \I_{+,0}$ on $G_{W(\kappa)\times_{W(\kappa)}}\I'$, and hence $(G'\times_{W(\kappa)}\I')_\kappa$ is mapped to $P_{+,0}\times_\kappa \I_{+,0}$, which is mapped to $\I_{+,0}$ via $\rho$. So by the universality of $\I'\stackrel{u}{\rightarrow}\I$, there is an unique $\rho':G'\times_{W(\kappa)}\I'\rightarrow \I'$ making the above diagram commutative. It is clearly an $\ES$-morphism. That $\rho'$ satisfies the association law follows similarly from Proposition \ref{dilatation}.

Now we prove that the morphism $\xymatrix@C=1.2cm{G'_\ES\times_\ES \I'\ar[r]^(0.55){(\rho', pr_2)} &\I'\times_\ES \I'}$ is an isomorphism. We consider the commutative diagram $$\xymatrix@C=1.2cm{
G'\times_{W(\kappa)} \I'\ar[r]^(0.55){(\rho', pr_2)}\ar[d]^{u_1\times u} &\I'\times_\ES \I'\ar[d]^{u\times u}\\
G_{W(\kappa)}\times_{W(\kappa)} \I\ar[r]^(0.6){(\rho, pr_2)} &\I\times_\ES \I.
}$$
The morphism $\rho\times pr_2$ is an isomorphism, and hence has an inverse $i$. Note that $\I'_0\times_{\ES_0}\I'_0$ is mapped to $\I_{+,0}\times_{\ES_0}\I_{+,0}$, which is mapped to $P_{+,0}\times_{\kappa}\I_{+,0}$ via $i$, so by the universality of $G'\times_{W(\kappa)} \I'\rightarrow G_{W(\kappa)}\times_{W(\kappa)} \I$, there is an unique morphism $i':\I'\times_\ES \I'\rightarrow G'\times_{W(\kappa)} \I'$ such that $(u_1\times u)\circ i'=i\circ (u\times u)$. The universality also implies that $i'\circ((\rho', pr_2))=id$.

For an $\ES$-scheme $T$ which is flat over $W(\kappa)$, and a morphism $(t_1,t_2):T\rightarrow \I\times_\ES \I$, such that $t_i(T\times_\ES \ES_0)\subseteq \I_{+,0}$, $i=1,2$, the universality of $u:\I'\rightarrow I$ implies that there is a unique morphism $(t'_1,t'_2):T\rightarrow \I'\times_\ES\I'$ such that $(t_1,t_2)=(u\times u)\circ(t'_1,t'_2)$. One applies this to $T=\I'\times \I'$ and finds immediately $((\rho', pr_2))\circ i'=id$.
\end{proof}
Let $v':\I_+\rightarrow \I'$ be the natural morphism which is equivariant with respect to their torsor structures. Let $v'_m:\I_{+,m}\rightarrow \I'_m$ be the reduction mod $p^m$.

Applying the functor $\W_m(-)$ to the commutative diagram in the above proposition, by Theorem \ref{greenberg functor} and its remarks, we have that $\W_m(\I')\rightarrow \W_m(\ES)$ is a torsor under $\W_m(G')=G'_m$, and a commutative diagram of $\W_m(\ES)$-morphisms$$\xymatrix@C=1.2cm{
G'_m\times_{W_m(\kappa)} \W_m(\I')\ar[r]\ar[d]^{\W_m(u_1)\times \W_m(u)}&\W_m(\I')\ar[d]^{\W_m(u)}\\
G_{W(\kappa),m}\times_{\W_m(\kappa)} \W_m(\I)\ar[r]^(0.65){\W_m(\rho)}& \W_m(\I).
}$$
\subsubsection{Constructing certain torsors over $\ES_0$}\label{construcion of torsors}

To do our job, we have to construct torsors over a Zariski cover of $\ES_0$, rather than over $\W_m(\ES)$.
We write $\mathcal{G}$ for the group $G_{W(\kappa)}$ (resp. $G'$, resp. $P_+$) and $\mathcal{T}^\circ$ for the $\mathcal{G}$-torsor $\I$ (resp. $\I'$, resp. $\I_+$) over $\ES$. By applying $\W_m(-)$, we get a $\W_m(\mathcal{G})$-torsor $\W_m(\mathcal{T}^\circ)$ over $\W_m(\ES)$.

Let $U^i=\Spec R^i, 1\leq i \leq r$ be an open covering of
$\ES$. Let $U^i_m=U^i\times W_{m}(\kappa)$,
$R^i_m=R^i\otimes_{W(\kappa)} W_{m}(\kappa)$, then by smoothness
of $U^i_m$, there is a homomorphism of $W_{m}(\kappa)$-algebras
$w^i_m:R^i_m\rightarrow W_{m}(R^i_0)$ such that $w^i_m$ induces
the identity map on $R^i_0$. We will fix once and for all a system
of $w^i_m$s such that $w^i_{m-1}$ is induced by $w^i_m$ mod $p^{m-1}$.
Note that $w^i_m$ is the same as a morphism $U^i_0\rightarrow
\W_{m}(U^i)\subseteq \W_{m}(\ES)$, which will also be denoted
by $w^i_m$. By pulling back the $\W_{m}(\mathcal{G})$-torsor
$\W_{m}(\mathcal{T}^\circ)$ along $w_i$, we get a
$\W_{m}(\mathcal{G})$-torsor $\mathcal{T}^i_m$ over $U^i_0$. Let
$U_0=\coprod\limits_{i=1}^r U_0^i$, then we get a
$\W_{m}(\mathcal{G})$-torsor $\mathcal{T}_m$ over $U_0$ by
putting together $\mathcal{T}^i_m$s.

The scheme $\mathcal{T}_m/U_0$ should be viewed as a kind of ``relative $\W_{m}$ functor''. More precisely, for an affine scheme $T=\Spec A$ over $U_0$, $\mathcal{T}_m(T)$ is the subset of elements $t\in (\coprod_{i=1}^rU_m^i)(W_{m}(A))$ such that the diagram $$\xymatrix@C=1.2cm{
\Spec W_{m}(A)\ar[r]\ar[d]&\coprod_{i=1}^r\mathcal{T}^\circ|_{U_m^i}\ar[d]\\
\Spec W_{m}(\prod_{i=1}^r R_0)\ar[r]& \coprod_{i=1}^rU_m^i.
}$$ is commutative, where the lower horizontal morphism is induced by the $w^i_m$s. This implies that $\mathcal{T}_{m+1}\rightarrow \mathcal{T}_m$ is a (trivial) torsor under $F^{m*}(\Omega^{1\vee}_{\mathcal{T}^\circ_0/\ES_0})\otimes (\prod_{i=1}^r R_0)$

We apply the above constructions to $\mathcal{T}^\circ=\I_+$ (resp. $\I'$) and get a $G'_m$-torsor (resp. $P_{+,m}$-torsor) over $U_0$. By abusing notations, the $G'_m$-torsor (resp. $P_{+,m}$-torsor) will be denoted by $\I_{+,m}$ (resp.~$\I'_m$).

\subsection[The morphism $\zeta_m$]{The morphism
$\zeta_m$}\label{construction of zeta}

We want a $G'_m$-equivariant morphism $\zeta_m:\I'_m\rightarrow G_m$ which induces the level $m$ stratification. The $\kappa$-scheme $\I'_m$ is smooth and affine, so to define the morphism $\zeta_m:\I'_m\rightarrow G_m$, it suffices to work with $\I'_m$-points. This means that we only need to determine how to attach an element in $G_m(T)$ to an element in $\I'_m(T)$ with $T$ affine and smooth over $U_0$. Let $T=\Spec A$ and $t:T\rightarrow \I'_m$ be a morphism. Then $t$ corresponds to a morphism $W_{m}(T)\rightarrow \I'$ such that the diagram
$$\xymatrix{
W_{m}(T)\ar[r]\ar[d] &\I'\ar[d]\\
W_{m}(U_0)\ar[r]& \ES.
}$$
is commutative. Here the lower horizontal morphism is induced by $w^i_{m}$s at the beginning of this section. This diagram lifts to a commutative diagram $$\xymatrix{
W(T)\ar[r]^{\widetilde{t}}\ar[d] &\I'\ar[d]\\
W(U_0)\ar[r]& \ES.
}$$
We remak that $W(A)$ is $p$-adically complete and $p$-torsion free.

We fix some notations first. Let $S_0$ be an $\mathbb{F}_p$-scheme and $S$ be a $\INT_p$-scheme equipped with an automorphism $\sigma:S\rightarrow S$ lifting the absolute Frobenius on $S_0$. For an $S$-scheme $T$, we write $\sigma^*T$ or $T^{(\sigma)}$ for the pull back of $T$ via $\sigma$. For a coherent $O_S$-module $N$ and an $O_S$-linear homomorphism~$f$, we write $\sigma^*f$ or $f^{(\sigma)}$ for the pull back of $f$ via $\sigma$. If $N=M\otimes O_S$ for some finitely generated $\INT_p$-module $M$, then there is an $O_S$-linear isomorphism $\xi:N\rightarrow \sigma^*N$ given by $m\otimes s\rightarrow m\otimes1\otimes s$. The inverse of $\xi$ is given by $m\otimes t\otimes s\rightarrow m\otimes \sigma(t)s$. It is easy to check that $\sigma(f)=\xi^{-1}f^{(\sigma)}\xi$.

We now follow \cite{crys rep and F-crys} Appendix A and \cite{displayandcrysdieu}. Let $\sigma$ be the Frobenius on $W(A)$, and $\nu$ be the Verschiebung. Note that $\nu$ is injective. We write $\nu^{-1}:\nu(W(A))\rightarrow W(A)$ for the map by taking preimages. The tuple $(W(A),\nu(W(A)), A, \sigma, \nu^{-1})$ is a lifting frame (see \cite{displayandcrysdieu} page 12 for the definition, and 2.2 for this statement).

Let $\mathbb{D}(\A_T)$ be the Dieudonn\'{e} crystal attached to $\A_T[p^\infty]$. There is a canonical isomorphism $\mathbb{D}(\A_T)(W(A))\cong \V\otimes W(A)$. The Frobenius on $\mathbb{D}(\A_T)(W(A))$ induces via the isomorphism a $\sigma$-linear map $\varphi:\V\otimes W(A)\rightarrow \V\otimes W(A)$, and a linear map $v:\V\otimes W(A)\rightarrow (\V\otimes W(A))^{(\sigma)}$. We also have $\varphi^\lin\circ v=p, v\circ \varphi^\lin=p$. Let $F^1\subseteq \V\otimes W(A)$ be the preimage of $\text{Lie}(\A_T)\subseteq \mathbb{D}(\A_T)(A)=\V\otimes A$, then by the proof of \cite{displayandcrysdieu} Proposition 3.15, $\frac{\varphi}{p}$ is well defined on $F^1$, and induces a surjection $ (\frac{\varphi}{p})^\lin:\sigma^*F^1\rightarrow \V\otimes W(A)$. Note that $v^\vee$ induces the Frobenius ${}^\vee\!\varphi$ on $\V^\vee\otimes W(A)$. One defines the Frobenius on $(\V\otimes W(A))^\otimes[\frac{1}{p}]$ to be the one induced by $\varphi$ on $\V\otimes W(A)$
and $\frac{{}^\vee\!\varphi}{p}$ on $(\V^\vee\otimes W(A))[\frac{1}{p}]$.
Let $\varphi^{\lin}:(\V\otimes W(A))^{(\sigma)}\rightarrow \V\otimes W(A)$ be the linearization of $\varphi$. Then for a normal decomposition $\V\otimes W(A)=L\oplus M$ (see \cite{displayandcrysdieu} page 12 for the definition) where $L$ is a direct summand of $\V\otimes W(A)$ lifting $\text{Lie}(\A_T)$, the map $(\frac{\varphi}{p})^\lin\oplus \varphi^\lin:(L\oplus M)^{(\sigma)}\rightarrow \V\otimes W(A)$ is an isomorphism.

Let $\widetilde{t}\in \I'(W(T))$ be as before. Let us still write $\widetilde{t}$ for its image in $\I$. Then $\widetilde{t}(L^1\otimes W(A))$ is a lifting of $\V^1\otimes A$, and $\widetilde{t}(L^1\otimes W(A))\bigoplus \widetilde{t}(L^0\otimes W(A))$ is a normal decomposition of $\V\otimes W(A)$. Let $g_{\widetilde{t}}$ be the composition of
$$\xymatrix{L\otimes W(A)\ar[r]^(0.45){\xi}&(L\otimes W(A))^{(\sigma)}\ar[r]^(0.45){\widetilde{t}^{(\sigma)}}&(\V\otimes W(A))^{(\sigma)}}$$
$$\xymatrix@C=0.2cm{(\V\otimes W(A))^{(\sigma)}=\widetilde{t}^{(\sigma)}((L^1\otimes W(A))^{(\sigma)})\bigoplus \widetilde{t}^{(\sigma)}((L^0\otimes W(A))^{(\sigma)})\ar[rrr]^(0.8){(\frac{\varphi}{p})^\lin\oplus \varphi^\lin}&&& \V\otimes W(A)
}$$
$\text{and }\ \ \ \ \ \ \ \ \ \ \ \ \ \ \V\otimes W(A)\stackrel{\widetilde{t}^{-1}}{\longrightarrow} L\otimes W(A).$
\begin{proposition}\label{g_t lies in G}
The element $g_{\widetilde{t}}\in \GL(L)(W(A))$ is in $G_{W(\kappa)}(W(A))$.
\end{proposition}
\begin{remark}\label{independence of tilda t}
Let $g_t$ be the image of $g_{\widetilde{t}}$ in $\GL(L)(W_{m}(A))$. Then $g_t$ depends only on $t$ and not on the choice of the lifting $\widetilde{t}$. Let $\overline{(\frac{\varphi}{p})^\lin}$ (resp. $\overline{\varphi^\lin}$) be the reduction mod $V_{m}(A)$ of $(\frac{\varphi}{p})^\lin$ (resp. $\varphi^\lin$). Then $g_t$ is the composition of
$$\xymatrix{L\otimes W_{m}(A)\ar[r]^(0.45){\xi}&(L\otimes W_{m}(A))^{(\sigma)}\ar[r]^(0.45){t^{(\sigma)}}&(\V\otimes W_{m}(A))^{(\sigma)}}$$
$$\xymatrix@C=0.3cm{(\V_{W_{m}(A)})^{(\sigma)}=t^{(\sigma)}((L^1_{W_{m}(A)})^{(\sigma)})\bigoplus t^{(\sigma)}((L^0_{W_{m}(A)})^{(\sigma)})\ar[rrr]^(0.8){\overline{(\frac{\varphi}{p})^\lin}\oplus \overline{\varphi^\lin}}&&& \V_{W_{m}(A)}
}$$
and $\V_{W_{m}(A)}\stackrel{t^{-1}}{\longrightarrow} L_{W_{m}(A)}$
\end{remark}
We need preparations to prove this statement.
Let $R_i$ be as before, $R$ be $\Pi_{i=1}^rR_i$ and $\widehat{R}$ be the $p$-adic completion of $R$. By smoothness of $R$, there is a homomorphism $\sigma':\widehat{R}\rightarrow \widehat{R}$ lifting the Frobenius of $R_0$. The tuple $(\widehat{R}, (p), R_0, \sigma',\frac{\sigma'}{p})$ is lifting frame, one can transfer statements for $W(A)$ without any problem. Among there statements, I would like to mention that there is a Frobenius $\varphi':\V\otimes \widehat{R}\rightarrow \V\otimes \widehat{R}$ and a Verschiebung $v':\V\otimes \widehat{R}\rightarrow (\V\otimes \widehat{R})^{(\sigma')}$ inducing a Frobenius on $(\V^\otimes\otimes \widehat{R})[\frac{1}{p}]$ which is still denoted by $\varphi'$. Moreover, by Proposition \ref{crystal vs module with connec}, there is a topologically quasi-nilpotent integrable connection $\nabla':\V_{\widehat{R}}\rightarrow \V_{\widehat{R}}\hat{\otimes}_{\widehat{R}} \hat{\Omega}^1_{\widehat{R}}$ such that $\sigma'^*\V_{\widehat{R}}\rightarrow \V_{\widehat{R}}$ is $\nabla'$-parallel.

We need to work with normal decompositions. Let $\mathbf{G}$ be the closed subscheme of $\mathbf{Aut}_{\widehat{R}}(\V\otimes\widehat{R})$ that respects $s_{\dr}$. Then $\mathbf{G}$ is a reductive group scheme over $\widehat{R}$. Let $\mathbf{P}$ be the stabilizer in $\mathbf{G}$ of $\V^1\otimes \widehat{R}\subseteq \V\otimes \widehat{R}$, and $\mathbf{U}$ be the subgroup of $\mathbf{G}$ acting trivially on $\V^1_{\widehat{R}}\oplus (\V_{\widehat{R}}/\V^1_{\widehat{R}})$. Then $\mathbf{P}$ is a parabolic subgroup of $\mathbf{G}$ with unipotent radical $\mathbf{U}$. Using terminologies in \cite{CIMK} 1.1.2, we have the following result.
\begin{lemma}
The filtration $\V^1\otimes \widehat{R}\subseteq \V\otimes \widehat{R}$ is $\mathbf{G}$-split.
\end{lemma}
\begin{proof}
Let $\mathbf{P}'$ be the stabilizer in $\mathbf{Aut}_{\widehat{R}}(\V\otimes\widehat{R})$ of $\V^1\otimes \widehat{R}\subseteq \V\otimes \widehat{R}$, and $\mathbf{U}'$ be its unipotent radical. Let $\chi:\mathbb{G}_m\rightarrow \mathbf{P}'/\mathbf{U}'$ be the cocharacter inducing the grading $\V^1_{\widehat{R}}\oplus (\V_{\widehat{R}}/\V^1_{\widehat{R}})$, then by \cite{CIMK} Lemma 1.1.1, we only need to prove that $\chi$ factors through $\mathbf{P}/\mathbf{U}$. There exists an fppf $\widehat{R}$-algebra $A$ such that $\I_+(A)\neq \emptyset$. Take a $t\in \I_+(A)$, then $\V^1_A\oplus t(L^0\otimes A)$ is a $\mathbf{G}_A$-splitting. So by \cite{CIMK} Lemma 1.1.1 again, the base change to $A$ of $\chi$ factors through $\mathbf{P}_A/\mathbf{U}_A$, but this implies that $\chi$ factors through $\mathbf{P}/\mathbf{U}$.
\end{proof}
Let $\V\otimes \widehat{R}=F^1\oplus F^0$ be normal decomposition induced by a cocharacter $\mu'$ of $\mathbf{G}$ (one can take, for example, $F^1=\V^1\otimes\widehat{R}$ and $F^0$ a complement of it induced by a cocharacter lifting $\chi$), and $\V^\vee\otimes \widehat{R}=(F^0)^\vee\oplus (F^1)^\vee$ be the induced normal decomposition. Let $(\V^\otimes\otimes\widehat{R})^0$ be the submodule generated by elements in $(F^1)^a\otimes (F^0)^b\otimes(F^{0\vee})^c\otimes (F^{1\vee})^d$ with $a=d$. Note that $(\V^\otimes\otimes\widehat{R})^0$ is the submodule of weight 0 in $\V^\otimes\otimes\widehat{R}$ with respect to $\mu'$.

\begin{lemma}\label{working over R hat}
The map $\varphi'$ takes integral values on $(\V^\otimes\otimes\widehat{R})^0$. In other words, $\varphi'$ induces a map $(\V^\otimes\otimes\widehat{R})^0\rightarrow\V^\otimes\otimes\widehat{R}$. The section $s_\dr\in \V^\otimes \otimes \widehat{R}$ lies in $(\V^\otimes\otimes\widehat{R})^0$, it is $\varphi'$-invariant and annihilated by $\nabla'$.
\end{lemma}
\begin{proof}
The map $\varphi'$ acts on $F^1$ and $F^0$ as $\varphi'$, and on $F^{1\vee}$ and $F^{0\vee}$ as $\frac{{}^\vee\!\varphi'}{p}$. We know that $\varphi'$ (resp. $\frac{{}^\vee\!\varphi'}{p}$, ${}^\vee\!\varphi'$) takes integral values on $F^0$ (resp. $F^{0\vee}$, $F^{1\vee}$), and that $\varphi'(F^1)\subset p (\V\otimes \widehat{R})$, so $\varphi'$ takes integral values on $(\V^\otimes\otimes\widehat{R})^0$, as $a=d$.

To see that $s_\dr\in (\V^\otimes\otimes\widehat{R})^0$, one only needs to notice that by definition of $\I$, $\mathbf{G}$ acts trivially on $s_\dr$, and hence $s_\dr$ is of weight 0 with respect to $\mu'$, which means that $s_\dr\in (\V^\otimes\otimes\widehat{R})^0$.

To prove the last statement, we need to use Faltings's deformation theory. Let $x$ be a closed point in $\Spec\widehat{R}$, and $\widehat{R}_x$ be the completion with respect to the maximal ideal defining $x$. Note that $x$ corresponds to a closed point of $U_0=\Spec R_0$. Let $k$ be the residue field of $x$, and $R_G$ be the completion at identity of $U_{-,W(k)}$. Then $\A[p^\infty]_{\widehat{R}_x}$ is a deformation of $\A[p^\infty]_x$, and induces an isomorphism $\widehat{R}_x\cong R_G$ by \cite{CIMK} Proposition 2.3.5. Moreover, $s_{\dr}=s_{\mathrm{cirs},x}\otimes 1$ over $\widehat{R}_x$.

Let $\widetilde{x}$ be a lifting of $x$, then $\V^1_{\widetilde{x}}\subseteq \V_{\widetilde{x}}$ gives the Hodge filtration. Let $\sigma'':R_G\rightarrow R_G$ be Frobenius on $W(k)$ and $p$-th power in indeterminants. The Dieudonn\'{e} module of $\A[p^\infty]_{\widehat{R}_x}$ is the tuple $(M,F,\varphi'',\nabla'')$, where $M=\V_{\widetilde{x}}\otimes R_G=\V_{R_G}$, $F=\V^1_{\widetilde{x}}\otimes R_G\subseteq M$ is the Hodge filtration, $\varphi''=u\circ \varphi_{x}$ with $u$ the universal element of $U_{-}$ and $\varphi_x$ the Frobenius on $\V_{\widetilde{x}}$, and $\nabla''$ the connection. We have $\nabla''(s_{\dr})=0$ by \cite{CIMK} 1.5.4. We also have $s_{\dr}\in (\V^\otimes \otimes R_G)^0$ and is $\varphi''$-invariant. Here the $\varphi''$-action on $\V^\otimes \otimes R_G[\frac{1}{p}]$ and $(\V^\otimes \otimes R_G)^0$ are constructed similarly. The only difference is that when defining $(\V^\otimes \otimes R_G)^0$, we fix an isomorphism $t:L\otimes R_G\rightarrow \V\otimes R_G$ which maps  $s\otimes 1$ to $s_{\dr}$ and $L^1\otimes R_G$ to $\V^1\otimes R_G$. The normal decomposition in this case is the one induced by $t$.

Let $\iota:\widehat{R}\rightarrow \widehat{R}_x=R_G$ be the injection. Then by Construction \ref{construction of connection}, $\nabla''=\iota^*\nabla'$, and hence $\nabla'(s_{\dr})=0$. Since $\iota\circ \sigma'$ and $\sigma''\circ \iota$ has the same reduction modulo $p$, by Construction \ref{frob over diff rings}, there is a canonical isomorphism $\varepsilon:\sigma''^*\iota^*(\V_{\widehat{R}})\rightarrow \iota^*\sigma'^*(\V_{\widehat{R}})$. The linearization of $\varphi''$ is given by $$\sigma''^*(\V_{R_G})=\sigma''^*\iota^*(\V_{\widehat{R}})\stackrel{\varepsilon}{\longrightarrow}
\iota^*\sigma'^*(\V_{\widehat{R}})\stackrel{\iota^*\varphi'^{\lin}}{\longrightarrow} \iota^*(\V_{\widehat{R}})=\V_{R_G}.$$
The description of $\varepsilon$ in Construction \ref{frob over diff rings} shows that it respects $s_{\dr}$, as $\nabla'(s_{\dr})=0$. But then $\varphi'$ respects $s_{\dr}$, as both $\varphi''$ and $\varepsilon$ respect $s_{\dr}$, and $x$ is an arbitrary point.
\end{proof}


\begin{proof}\emph{(of Proposition \ref{g_t lies in G})}
We use notations as before Proposition \ref{g_t lies in G}. We will show that $g_{\widetilde{t}}\in \GL(L)(W(A))$ fixes $s_{\dr}$. The homomorphism $\widehat{R}\rightarrow W(A)$ induces $R_0\rightarrow W(A)/(p)\rightarrow A$, and hence canonical isomorphisms $$\V_{W(A)}\cong \mathbb{D}(\A_{W(A)/(p)})(W(A))\cong \mathbb{D}(\A_A)(W(A)).$$

Let $F^i=\widetilde{t}(L^i\otimes W(A))$, then $\V_{W(A)}=F^1\oplus F^0$ is a normal decomposition of $\V_{W(A)}$. Moreover, the splitting is induced by a cocharacter of $\mathbf{G}_{W(A)}$. Let $(\V^\otimes_{W(A)})^0$ be the submodule of weight 0 in $\V^\otimes_{W(A)}$. Then by the same argument as in the proof of Lemma \ref{working over R hat}, $\varphi$ takes integral value on $(\V^\otimes_{W(A)})^0$, and that $\varphi^\lin:(\V^\otimes_{W(A)})^{0,(\sigma)}\rightarrow \V^\otimes_{W(A)}$ is the same as the restriction to $(\V^\otimes_{W(A)})^{0,(\sigma)}$ of $(\frac{\varphi}{p})^\lin\oplus \varphi^\lin=\widetilde{t}g_{\widetilde{t}}\xi^{-1}\widetilde{t}^{-1,(\sigma)}$.

So we reduce to check that $s_\dr$ is $\varphi$-invariant. Let $\iota:\widehat{R}\rightarrow W(A)$ be the composition of $w:\widehat{R}\rightarrow W(R_0)$ induced by the systems of morphisms $w^i_m$s at the beginning of \ref{construcion of torsors} and the natural homomorphism $W(R_0)\rightarrow W(A)$. Noting that $W(A)$ is $p$-adically complete and $p$-torsion free, and that $\iota\circ \sigma'=\sigma\circ \iota\ \mathrm{mod}\ p$, the statement follows from the same argument as in the second half of the proof of Lemma \ref{working over R hat}, by using that $s_{\dr}$ is $\varphi'$-invariant and annihilated by $\nabla'$.
\end{proof}

Applying Proposition \ref{g_t lies in G} and the remark after it to $A$ such that $\Spec A=\I'_m$, we get a morphism $\zeta_m:\I'_m\rightarrow G_m$, given by mapping $t\in \I'_m(\I'_m)$ to $g_t$.
\begin{proposition}
\

(1) The morphism $\zeta_m$ is $G'_m$-equivariant considering left $G'_m$-actions, and hence induces a morphism $\zeta_{m,\#}:U_0\rightarrow [G'_m\backslash G_m]$.

(2) The level $m$ stratification on $U_{0,\overline{\kappa}}$ is induced by fibers of $\zeta_{m,\#}\otimes \overline{\kappa}$, and hence each stratum on $U_{0,\overline{\kappa}}$ is locally closed.
\end{proposition}
\begin{proof}
For 1, we have a morphism $\I'_m\rightarrow \I_m$, which is equivariant with respect to $G'_m\rightarrow G_m$. Here the left action of $h$ in $G'_m$ (resp. $G_m$) on $\I'_m$ (resp. $\I_m$) is given by $h\cdot t=t\circ h^{-1}$. We use notations as in Remark \ref{independence of tilda t}. We have $(h\cdot t)^{(\sigma)}\circ \xi=t^{(\sigma)}\circ h^{-1,(\sigma)}\circ \xi=(t^{(\sigma)}\circ \xi)\circ (\xi^{-1}\circ h^{-1,(\sigma)}\circ \xi)=(t^{(\sigma)}\circ \xi)\circ \sigma(h)^{-1}$. Then the correspondence $t\mapsto t\circ h^{-1}$ induces the correspondence $g_t\mapsto hg_t\sigma(h)^{-1}$. This proves 1.

For 2, by definition of level $m$-stratification on $U_0$, we work with algebraically closed fields. Let $k$ be such a field, and $x,y$ be two $k$-points of $U_0$. Then $x$ and $y$ are in the same stratum if there is an isomorphism of Dieudonn\'{e} modules $\rho:\mathbb{D}(\A_x[p^\infty])(W_{m}(k))\cong\mathbb{D}(\A_y[p^\infty])(W_{m}(k))$ mapping $s_{\cris,x}$ to $s_{\cris,y}$. Let $t_x$ (resp. $t_y$) be a section of $\I'_{m,x}$ (resp. $\I'_{m,y}$). The constructions as in Proposition \ref{g_t lies in G} give $g_{t_x}$ (resp. $g_{t_y}$) in $G_m(k)$. Let $\varphi_x$ (resp. $\varphi_y$) be the Frobenius, then $\varphi_y=\rho\circ \varphi_x(\rho^{-1}\cdot)$. There is an $h\in G'_m(k)$ such that $t_y=h(\rho\circ t_x)=\rho\circ t_x \circ h^{-1}$. The formulas for $\varphi_y$ and $t_y$ imply immediately that $g_{t_y}=hg_{t_x}\sigma(h)^{-1}$.

Conversely, if for $x,y,t_x,t_y$ as above, there exists an $h\in G'_m(k)$ such that $g_{t_y}=hg_{t_x}\sigma(h)^{-1}$, then $\rho=(h^{-1}\cdot t_y)\circ t_x^{-1}$ is an isomorphism of Dieudonn\'{e} modules   $\mathbb{D}(\A_x[p^\infty])(W_{m}(k))\cong\mathbb{D}(\A_y[p^\infty])(W_{m}(k))$ mapping $s_{\cris,x}$ to $s_{\cris,y}$.

Let $o$ be a point in the topological space of $[G'_m\backslash G_m]$. Noting that $o$ is locally closed in $[G'_m\backslash G_m]$ and that it admits a representative over $\overline{\kappa}$, we know that $\zeta_{m,\#}^{-1}(o)\subseteq U_{0,\overline{\kappa}}$ is locally closed. This proves 2.
\end{proof}
\begin{theorem}\label{smoothness}
The morphism $\zeta_{m,\#}$ is smooth.
\end{theorem}
\begin{proof}
By faithfully flat descent, it is equivalent to that $\zeta_m$ is smooth. By \cite{AG} III Proposition 10.4, we only need to show that for any $\overline{\kappa}$-point $t$ of $\I'_m$, the induced map on tangent spaces $T_{\zeta_m}:T_t\I'_m\rightarrow T_{\zeta_m(t)}G_m$ is surjective. The element $t$ (resp. $\zeta_m(t)$) corresponds to an element in $\I'(W_{m}(\overline{\kappa}))$ (resp. $G(W_{m}(\overline{\kappa}))$) which is still denoted by $t$ (resp. $\zeta_m(t)$). Note that $t$ lifts to an element $\widetilde{t}\in \I'(W(\overline{\kappa}))$ with image $g_{\widetilde{t}}\in G(W(\overline{\kappa}))$ which is a lift of $\zeta_m(t)$.

Let $\overline{\kappa}[\varepsilon]$ be such that $\varepsilon^2=0$. We will prove the following statement which implies the subjectiveness of $T_{\zeta_m}$.

$(*)$  for any $ g[\varepsilon]\in G(W(\overline{\kappa}[\varepsilon]))$ whose image in $G(W(\overline{\kappa}))$ deforms $g_{\widetilde{t}}$, there is a $\widetilde{t}[\varepsilon]\in \I'(W(\overline{\kappa}[\varepsilon]))$ deforming $\widetilde{t}$, such that $g_{\widetilde{t}[\varepsilon]}=g[\varepsilon]$.

Let $\pi:\I'_m\rightarrow U_0$ be the projection. We fixed a homomorphism $O_U\rightarrow W(O_{U_0})$, it induces a homomorphism $\iota:R_G\cong O_{U,\pi(t)}^\wedge\rightarrow W(O_{U_0,\pi(t)}^\wedge)\cong W(R_{G,0})$. Here $(\cdot)^\wedge$ means the completion with respect to the ideal defining $\pi(t)$. Viewing $\widetilde{t}$ as an $R_G$-point, we get a trivialization $G'\times_{W(\overline{\kappa})}R_G\stackrel{\simeq}{\longrightarrow}\I'\times_U R_G$ of $\I'$, $h\mapsto h\cdot \widetilde{t}$. Using notations as in the proof of Lemma \ref{working over R hat}, by Faltings's deformation theory, the construction in \ref{g_t lies in G} gives $ug_{\widetilde{t}}\in G(R_G)$ when applied to $\widetilde{t}$ and $\V_{R_G}$, and $hug_{\widetilde{t}}\sigma''(h)^{-1}$ when applied to $h\circ \widetilde{t}$.

Let $\sigma$ be the Frobenius on $W(R_{G,0})$. Let $R_G(2)$ be $p$-adic completion of the PD-envelope of $R_G\widehat{\otimes} R_G$ with respect to the $\mathrm{ker}(R_G\widehat{\otimes} R_G\rightarrow R_{G,0})$, and $K$ be $\mathrm{ker}(R_G\widehat{\otimes} R_G\rightarrow R_{G})$. The Dieudonn\'{e} crystal structure of $\mathbb{D}(\A_{R_{G,0}})$ induces an isomorphism $$\theta:\big(R_G(2)/K^{[2]}\big)\otimes_{R_G} \V_{R_G}\rightarrow \V_{R_G}\otimes_{R_G} \big(R_G(2)/K^{[2]}\big).$$ There are also canonical isomorphisms
$$
\xymatrix@C=0.33cm{\V_{R_G}\otimes_{R_G} R_G(2)\ar[r]^(0.45)\alpha_(0.45){\simeq}&\V_{W(\overline{\kappa})}\otimes_{W(\overline{\kappa})}R_G(2)\ar[r]^i_{\simeq}&R_G(2)
\otimes_{W(\overline{\kappa})}\V_{W(\overline{\kappa})}\ar[r]^(0.55)\beta_(0.55){\simeq}
&R_G(2)\otimes_{R_G}\V_{R_G}.}
$$
By \cite{CIMK} 1.5.4, the composition $\theta_{\widetilde{t}}=\widetilde{t}^{-1}\circ\alpha\circ\theta\circ\beta\circ \widetilde{t}$ is an element in $U_-(R_G(1))$ via the identification~$i$.

Now by Construction \ref{frob over diff rings}, the morphism $\zeta$ on $W(R_{G,0})$-points is given by $$h\mapsto hug_{\widetilde{t}}\xi^{-1}((\sigma\iota\cdot \iota\sigma'')^*\theta_{\widetilde{t}})\xi\sigma(h)^{-1}, \ \ h\in G'(W(R_{G,0})).$$
Let $\widetilde{t}[\varepsilon]$ be a deformation of $\widetilde{t}$. Then its image gives a deformation $\pi(\widetilde{t}[\varepsilon])$ of $\pi(\widetilde{t})$. For $h\in\mathrm{ker}(G'(W(\overline{\kappa}[\varepsilon]))\rightarrow G'(W(\overline{\kappa})))$, we have $\sigma(h)=\mathrm{id}$, as $\sigma(W(\overline{\kappa}[\varepsilon]))\subseteq W(\overline{\kappa})$. Similarly, we have $(\xi^{-1}((\sigma\iota\cdot \iota\sigma'')^*\theta_{\widetilde{t}}\xi)_{\pi(\widetilde{t}[\varepsilon])}=\mathrm{id}$  as $((\sigma\iota\cdot \iota\sigma'')^*\theta_{\widetilde{t}})_{\pi(\widetilde{t})}=\mathrm{id}$. The multiplication morphism $P_+\times U_-\rightarrow G$ is \'{e}tale at $(1,1)\in P_+\times U_-$, so by the description of $G'$, the morphism $G'\times R_G\rightarrow G$, $(h,u)\mapsto hu$ is smooth at $(1,1)$. This implies $(*)$ and hence the theorem.
\end{proof}

We will use the morphism $\zeta_{m,\#}$ to study level $m$
stratifications. Before stating the main theorem, we introduce the
following terminology.

\begin{definition}
An immersion of schemes $S\rightarrow T$ is said to be pure if it is an affine morphism.
\end{definition}



\begin{theorem}\label{main theorem level m--locally closed+sm}
\

(1) Each level $m$ stratum is locally closed in $\ES_{0,\overline{\kappa}}$. Moreover, (given the reduced induced scheme structure) it is smooth and equi-dimensional.

(2) The Zariski closure of a level $m$ stratum is a union of level $m$ strata.

(3) Each level $m$ stratum is and pure in $\ES_{0,\overline{\kappa}}$.
\end{theorem}
\begin{proof}
For (1), let $p:U_{0,\overline{\kappa}}\rightarrow \ES_{0,\overline{\kappa}}$ be the natural morphism, then two geometric points of $U_{0,\overline{\kappa}}$ lie in the same level $m+1$ stratum if and only if their images in $\ES_{0,\overline{\kappa}}$ are in the same stratum. This implies that a stratum of $\ES_{0,\overline{\kappa}}$ is locally closed with respect to an open covering, and hence locally closed in $\ES_{0,\overline{\kappa}}$. Let $\ES_{0,\overline{\kappa}}^c\subseteq \ES_{0,\overline{\kappa}}$ be a level $m+1$-stratum, then $p^{-1}(\ES_{0,\overline{\kappa}}^c)$ is smooth by the previous theorem. This implies that $\ES_{0,\overline{\kappa}}^c$ is smooth.

For (2), recall that $p:U_{0,\overline{\kappa}}\rightarrow \ES_{0,\overline{\kappa}}$ is the natural morphism $\coprod U_{0,\overline{\kappa}}^i\rightarrow \ES_{0,\overline{\kappa}}$, where $\coprod U_{0,\overline{\kappa}}^i$ is a finite open covering of $\ES_{0,\overline{\kappa}}$. So for a subset $S\subseteq \ES_{0,\overline{\kappa}}$, we have $\overline{p^{-1}(S)}=\coprod (\overline{S}\bigcap U_{0,\overline{\kappa}}^i)=p^{-1}(\overline{S})$. We only need to show that the closure of a stratum in $U_{0,\overline{\kappa}}$ is a union of strata, as $p$ is surjective. Let $c'$ be a point in the topological space of $[G'_m\backslash G_m]\otimes \overline{\kappa}$. Let $C\subseteq |[G'_m\backslash G_m]\otimes \overline{\kappa}|$ be the subset of points that generalize to $c$, then the universally openess of $\zeta_\#$ implies that $\overline{U_{0,\overline{\kappa}}^c}=\bigcup\limits_{c'\in C}U_{0,\overline{\kappa}}^{c'}$.

For (3), we only need to show that the immersion $U_{0,\overline{\kappa}}^c\rightarrow U_{0,\overline{\kappa}}$ is affine. It suffices to show that $O^c_m$, the $G'_{m,\overline{\kappa}}$-orbit in $G_{m,\overline{\kappa}}$, is an affine scheme. Let's write $c$ for its image in $[G'_0\backslash G_0]\otimes \overline{\kappa}$ and $O_0^c$ for the corresponding $G'_{0,\overline{\kappa}}$-orbit. Then by \cite{purity} Theorem 2.2, $O_0^c$ is affine. The same argument as in the proof of \cite{purity-N-V-W} Lemma 5.2 implies that $O^c_m$ is affine.
\end{proof}

\subsection[Independence of symplectic emdeddings]{Independence of symplectic emdeddings}

In this subsection, we prove that level $m$ stratifications are
determined by Shimura data. The method here is a variation of the
one in \cite{remarkZ}. The definition of level $m$ stratifications
is independent of choices of open covers, so we only need to show
that the morphism $\zeta_{m,\#}:U_0\rightarrow [G'_m\backslash
G_m]$ is determined by its Shimura datum.

We first check that the stack $[G'_m\backslash G_m]$ is determined
by its Shimura datum. The pair $(G,X)$ determines a unique
$G_{\INT_p}(W(\kappa))$-conjugacy class of cocharacters that $\mu$
is one of them. The construction of $G'_m$ depends only on $\mu$,
and we use $G'^\mu_{m}$ to indicate the cocharacter.
\begin{lemma}
Let $\nu:\GG_{m,W(\kappa)}\rightarrow G_{W(\kappa)}$ be a
cocharacter conjugate to $\mu$, then there is a canonical
isomorphism $i_{m}^{\mu,\nu}:[G'^\mu_{m}\backslash
G_{m}]\stackrel{\simeq}{\longrightarrow}[G'^\nu_{m}\backslash
G_{m}]$.
\end{lemma}
\begin{proof}
Let $g^0\in G_{W(\kappa)}(W(\kappa))$ be such that
${}^{g^0}\!\mu:=g^0\mu g^{0,-1}=\nu$. Then ${}^{g^0}\!P_+^\mu=P_+^\nu$.
The composition $G'^\mu\rightarrow
G_{W(\kappa)}\stackrel{{}^{g^0}\!(-)}{\longrightarrow}
G_{W(\kappa)}$ is such that its special fiber factors through
$P_+^\nu$. So by the universality of dilatation, we have a unique
homomorphism $G'^\mu\rightarrow G'^\nu$, which is necessarily an
isomorphism, making the diagram
$$\xymatrix{
G'^\mu\ar[r]\ar[d] &G'^\nu\ar[d]\\
G_{W(\kappa)}\ar[r]^{{}^{g^0}\!(-)}& G_{W(\kappa)}}$$ commutative.

Let $i^{g^0}:G_{W(\kappa)}\rightarrow G_{W(\kappa)}$ be the
morphism such that $g\mapsto g^0g\sigma(g_0)^{-1}$. Consider the
$G'^\mu$-action (resp. $G'^\nu$-action) on the source (resp.
target) of $i^{g_0}$ given by the natural morphism
$G'^\mu\rightarrow G_{W(\kappa)}$ (resp. $G'^\nu\rightarrow
G_{W(\kappa)}$). Direct computation shows that $i^{g^0}$ is
equivariant with respect to $G'^\mu\rightarrow G'^\nu$, and hence
induces an isomorphism of stacks $i_m^{g^0}:[G'^\mu_{m}\backslash
G_{m}]\stackrel{\simeq}{\longrightarrow}[G'^\nu_{m}\backslash
G_{m}]$ after applying the Greenberg functor.

The morphism $i_m^{g^0}$ is independent of choices of $g^0$. As, if $g^1\in G_{W(\kappa)}(W(\kappa))$ is another element
such that ${}^{g^1}\!\mu=\nu$, then there is a unique $l\in
L^\mu(W(\kappa))$, such that $g^1=g^0l$. But
$i^{g^1}(g)=g^1g\sigma(g^1)^{-1}=g^0lg\sigma(l)^{-1}\sigma(g^0)^{-1}=i^{g^0}(l\cdot
g)$, so they induce the same morphism on stacks.
\end{proof}
\begin{remark}
By canonical, we mean that for three cocharacters
$\mu,\nu,\lambda$ conjugating to each other, there is the identity
$i_{m}^{\mu,\lambda}=i_{m}^{\nu,\lambda}\circ i_{m}^{\mu,\nu}$.
\end{remark}
\begin{proposition}
The morphism $\zeta_{m,\#}:U_0\rightarrow [G'_m\backslash G_m]$
depends only on the Shimura datum, not on the choice of sympletic
embedding, the lattice and the Hodge-Tate tensor.
\end{proposition}
\begin{proof}
The proof is extracted from \cite{remarkZ}. Let $i_1:(G,X)\hookrightarrow
(\GSp(V_1,\psi_1))$ and $i_2:(G,X)\hookrightarrow
(\GSp(V_2,\psi_2))$ be two symplectic embeddings. There is a symplectic embedding
$i:(G,X)\hookrightarrow (\GSp(V,\psi))$ with $V=V_1\oplus V_2$.
Let $L_1\subseteq V_1$ and $L_2\subseteq
V_2$ be $\INT_{(p)}$-lattices, $s_1\in L_1^\otimes$ and $s_2\in
L_2^\otimes$ be tensors defining $G_{\INT_{(p)}}$. Let $L$ be $L_1\oplus L_2$, and $s\in L^\otimes$ be
a tensor defining $G_{\INT_{(p)}}\subseteq \GL(L)$. Let $\A_i$ be the universal abelian scheme restricted to $\ES$, and $\A$ be the one induced by $i$. Then $\A=\A_1\times \A_2$. One gets $\V$, $\V_1$, $\V_2$ by taking de Rham cohomology, with Hodge filtrations denoted by $\Fil^1()$.

Let $\I=\IIsom\big((L^\vee,s)\otimes O_{\ES}, (\V, s_{\dr})\big)$, $\I_1:=\IIsom\big((L_1^\vee, s_1)\otimes O_{\ES},
(\V_1,s_{1,\dr})\big)$, and $\mathbb{J}:=\IIsom\big((L^\vee,L_1^\vee,
s_1,s)\otimes O_{\ES}, (\V,\V_1,s_{1,\dr}, s_{\dr})\big)$, the scheme of isomorphisms that maps $L_1^\vee$ to $\V_1$ and respects the Hodge-Tate tensors. Then the natural morphisms $\I\leftarrow\mathbb{J}\rightarrow \I_1$ are $G_{W(\kappa)}$-equivariant isomorphisms. Similarly, we have a natural isomorphism of $P_+$-torsors $\I_+\rightarrow \I_{1,+}$. By the universality of dilatations, we get an isomorphism $\I'\rightarrow \I'_{1}$. By applying the Greenberg functor and pulling back to $U_0$, we get an isomorphism of $G'_m$ torsors $\I'_m\rightarrow \I'_{1,m}$.

To deduce the proposition, one only needs to check that the diagram
$$\xymatrix{\I'_m\ar[rr]\ar[dr]^{\zeta_m}  &&\I'_{1,m}\ar[dl]_{\zeta_{1,m}}\\
&G_m}$$
is commutative. But this follow from that $\V_{\widehat{R}}=\V_{1,\widehat{R}}\times\V_{2,\widehat{R}}$ as Dieudonn\'{e} modules, and that the $\zeta_m$s are constructed using $F$-$V$-module structures.
\end{proof}
\begin{remark}As a direct consequence, we see level $m$ stratifications are
independent of choices of symplectic embeddings.
\end{remark}

\section[Truncated displays with additional structure]{Truncated displays with additional structure}

\subsection[Truncated displays]{Truncated displays}
Let $R$ be a commutative ring of characteristic $p$, $W(R)$ be the
ring of Witt vectors and $I_{m+1}(R)$ be
$\mathrm{ker}(W_{m+1}(R)\rightarrow R)$. The inverse of
Verschiebung induces a $\sigma$-linear bijection
$v^{-1}:I_{m+1}(R)\rightarrow W_m(R)$. Let $J_{m+1}(R)$ be
$\mathrm{ker}(W_{m+1}(R)\rightarrow W_m(R))$.

\begin{definition}
A truncated pair of level $m$ over $R$ is a tuple
$(P,Q,\iota,\epsilon)$. Here $P$ is a projective $W_m(R)$-module
of finite rank, $Q$ is a finitely generated $W_m(R)$-module,
$\iota:Q\rightarrow P$ and
$\epsilon:I_{m+1}(R)\otimes_{W_m(R)}P\rightarrow Q$ are
homomorphisms. The following conditions are required:

(1) The compositions $\iota\epsilon$ and $\epsilon(\iota\otimes
1)$ are the multiplication maps.

(2) $\mathrm{coker}(\iota)$ is a finite projective $R$-module.

(3) $\epsilon$ induces an exact sequence
$$
\xymatrix{0\ar[r]&J_{m+1}(R)\otimes \mathrm{coker}(\iota)\ar[r]&
Q\ar[r]^{\iota}& P\ar[r]& \mathrm{coker}(\iota)\ar[r]&0}.
$$
\end{definition}

\begin{definition}
A normal decomposition of a truncated pair consists of projective
$W_{m}$-modules $L\subseteq Q$ and $T\subseteq P$ such that we
have isomorphisms $\xymatrix@C=0.7cm{L\oplus
T\ar[r]^(0.6){\iota\oplus \mathrm{id}} &P}$ and
$\xymatrix@C=0.7cm{L\oplus (I_{m+1}(R)\otimes_{W_m(R)}
T)\ar[r]^(0.8){\mathrm{id}\oplus \epsilon} &Q}$.
\end{definition}
\begin{lemma}
Every truncated pair admits a normal decomposition.
\end{lemma}
\begin{proof}
This is \cite{smooth truncated display} Lemma 3.3.
\end{proof}
\begin{definition}
A truncated display of level $m$ over $R$ is a
tuple $(P,Q,\iota,\epsilon,F,V^{-1})$. Here
$(P,Q,\iota,\epsilon)$ is a truncated pair, $F:P\rightarrow P$ and
$V^{-1}:Q\rightarrow P$ are $\sigma$-linear maps such that
$V^{-1}\epsilon=v^{-1}\otimes F$ and that the image of $V^{-1}$
generates $P$ as a $W_m(R)$-module.
\end{definition}
For a truncated pair $P,Q,\iota,\epsilon$ with normal
decomposition $(L,T)$, the set of pairs $(F,V^{-1})$ is in
bijection with the set of $\sigma$-linear isomorphism
$\Psi:L\oplus T\rightarrow P$ such that $\Psi|_L=V^{-1}|_L$ and
$\Psi|_T=F|_T$. If $L$ and $T$ are free $W_m(R)$-modules, $\Psi$
is described by an invertible matrix over $W_m(R)$. The triple
$(L,T,\Psi)$ is called a normal representation of
$(P,Q,\iota,\epsilon,F,V^{-1})$.

We recall the mains of truncated displays in \cite{smooth
truncated display}. We will fix a positive integer $h$. Let
$\mathscr{B}\mathscr{T}_m$ be the category of BT-$m$s of
height $h$ fibered over the category of affine
$\mathbb{F}_p$-schemes, and $\mathscr{D}isp_m$ be the category of
truncated displays of level $m$ and rank $h$ fibered over the
category of affine $\mathbb{F}_p$-schemes. Taking Dieudonn\'{e}
display of a $p$-divisible group induces a morphism
$\Phi_m:\mathscr{B}\mathscr{T}_m\rightarrow \mathscr{D}isp_m$.

Let $R$ be an $\mathbb{F}_p$-algebra, and $X$ be a BT$\text{-}m$ over
$R$. We write $\AAut(X)$ (resp. $\AAut(\Phi_m(X))$) for the
automorphism group scheme of $X$ (resp. $\Phi_m(X)$), and
$\AAut^o(X)$ for $\mathrm{ker}(\AAut(X)\rightarrow\AAut(\Phi_m(X)))$.

\begin{theorem}\label{Lau's main results}
\

(1) The fiber categories $\mathscr{B}\mathscr{T}_m$ and
$\mathscr{D}isp_m$ are both smooth algebraic stacks of dimension 0
over~$\mathbb{F}_p$.

(2) The morphism $\Phi_m$ is smooth.

(3) The group scheme $\AAut^o(X)$ is commutative infinitesimal
finite flat over $R$. The natural morphism $\IIsom(X,Y)\rightarrow
\IIsom(\Phi_m(X),\Phi_m(Y))$ is a torsor under $\AAut^o(X)$.
\end{theorem}
\begin{proof}
In 1, the statement for $\mathscr{B}\mathscr{T}_m$ is proved in
\cite{dimw}, that for $\mathscr{D}isp_m$ is \cite{smooth truncated
display} Proposition 3.15. The second statement is \cite{smooth
truncated display} Theorem 4.5, and the third statement is
\cite{smooth truncated display} Theorem 4.7.
\end{proof}
We recall Lau's construction realizing $\mathscr{D}isp_m$ as a
disjoint union of quotient stacks. For an integer $d$ with $0\leq
d\leq h$, let $\mathscr{D}isp^d_m$ be the substack such that
$\mathrm{coker}(\iota)$ has rank $d$. Let $G_m$ be the presheaf on
affine $\mathbb{F}_p$-schemes such that $G_m(R)$ is the set of
$h\times h$ invertible $W_m(R)$-matrices. Then $G_m$ is
represented by an affine smooth $\mathbb{F}_p$-scheme. There is a
morphism $\pi_{m,d}:G_m\rightarrow \mathscr{D}isp^d_m$ such that
$\pi_{m,d}(g)$ is given by the normal decomposition $(L, T, \Psi)$
where$L=W_m(R)^{h-d}$, $T=W_m(R)^d$ and $\Psi$ has matrix
representation $g$. Let $G_{m,d}$ be the sheaf of groups such that
$G_{m,d}(R)$ is the set of matrices $\bigl( \begin{smallmatrix} A
& B \\ C & D
\end{smallmatrix} \bigr)$ with
$A\in \mathrm{Aut}(L)$, $B\in \mathrm{Hom}(T,L)$, $C\in
\mathrm{Hom}(L,I_{m+1}(R)\otimes T)$ and $D\in \mathrm{Aut}(T)$.
Then $G_{m,d}$ is also an affine smooth $\mathbb{F}_p$-scheme.
Moreover, $[G_{m,d}\backslash G_m]\cong \mathscr{D}isp^d_m$.

\subsection[Truncated displays with additional structure]{Truncated displays with additional structure}
We start with the datum $(G_{W(\kappa)},\mu)$. For the embedding
$G_{W(\kappa)}\subseteq \GL(L_{W(\kappa)})$, let $L^1\subseteq
L^\vee_{W(\kappa)}$ (resp. $L^0\subseteq L^\vee_{W(\kappa)}$) be
the sub-module of weight 1 (resp. 0). For a $\kappa$-algebra $R$,
let $X_m(R)$ be the set of $\sigma$-linear isomorphisms
$\Psi:(L^0\oplus L^1)\otimes W_m(R)$, such that

(1) $(L^1_{W_m(R)}, L^0_{W_m(R)}, \Psi)$ is a normal
representation of a truncated display structures on
$L^\vee_{W(\kappa)}\otimes W_m(R)$.

(2) $\sdr\otimes 1\in L_{W(\kappa)}^\otimes\otimes W_m(R)$ is
$F$-invariant.

Our arguments before imply that $X_m(R)\cong G_m(R)$. We say two
elements in $G_m(R)$ are equivalent if and only if there is an
isomorphism between their corresponding truncated displays
respecting the tensor $\sdr\otimes 1$.

Let $KU_{-,m}$ be $\mathrm{ker}(U_{-,m+1}\rightarrow U_{-,0})$. It
represents the functor which associates to a
$\kappa$-algebra $R$ the set $1+I_{m+1}(R)\otimes
\mathrm{Lie}(U_-)$. Let $KG'_m$ be $U_{+,m}\times L_{m}\times
KU_{-,m}$ whose group structure is such that
$(h_1,h_2,h_3)=(h'_1,h'_2,h'_3)\cdot (h''_1,h''_2,h''_3)$ if and
only if $h_1h_2h_3=h'_1h'_2h'_3h''_1h''_2h''_3$. The group $KG'_m$
acts on $G'_m$ via $\sigma$-conjugation, and two elements in
$G'_m(R)$ are equivalent if and only if they are in the same
$KG'_m(R)$-orbit. We remark that if $G=\GL_{2g}$, and $\mu$ a
cocharacter defined over $\INT_p$ whose subspace of weight 1
(resp. 0) is of rank $g$, then our $KG'_m$ here is precisely
$G_{m,g}$ in the previous subsection.

There is a natural homomorphism $G'_m\rightarrow KG'_m$ which is
identity on $U_{+,m}$ and $L_{m}$, and the morphism induced by
``multiplication by $p$'' $$p:W_m(R)\rightarrow I_{m+1}(R),\ \ \
(r_1,r_2,\cdots, r_m)\mapsto(0,r_1^p,r^p_2,\cdots,r^p_m)$$ on
$U_{-,m}$. This homomorphism is faithfully flat, bijective on
geometric points, and has finite kernel. Moreover, this
homomorphism is compatible with their actions on $G_m$, i.e. the
action of $g'\in G'_m$ on $g\in G_m$ is the same as that of its
image in $KG'_m$ on $g$. This induces a morphism of stacks
$[G'_m\backslash G_m]\rightarrow[KG'_m\backslash G_m]$ which is
smooth and bijective on geometric points.

Let $\ES$, $U^i=\Spec R^i, 1\leq i \leq r$, $U_0$, $\I'_m/U_0$ and
$\zeta_m:\I'_m\rightarrow G_m$ be as before. We will
construct a morphism $\ES_0\rightarrow [KG'_m\backslash G_m]$
whose geometric fibers are level $m$ strata. In fact, we will show
that $\I'_m$ extends to a $KG'_m$-torsor $K\I'_m$ over $U_0$ which
descents to $\ES_0$, and the morphism $\zeta_m$ extends to a
$KG'_m$-equivariant morphism $K\zeta_m:K\I'_m\rightarrow G_m$
which also descents to $\ES_0$.

The constructions of $K\I'_m$ and $K\zeta_m$ are straight forward:
one takes $K\I'_m$ to be $\I_m\times^{G'_m}KG'_m$ and $K\zeta_m$
to be $\zeta_m\times^{G'_m}\mathrm{id}_{KG'_m}$. Here we use the
easy fact the the morphism $G_m=G_m\times^{G'_m}G'_m\rightarrow
G_m\times^{G'_m}KG'_m$ is an isomorphism. We need a better
description of $\I'_m\rightarrow K\I'_m$. We describe $\I'_m$
first. It is regular, so it suffices to work with $\I'_m(A)$ with
$A/U_0$ regular. Noting that $W(A)$ is $p$-torsion free, an
element $t\in \I'_m(A)$ lifts to a $W(A)$-point $\widetilde{t}$ of
$\I'$, where $\widetilde{t}$ is an isomorphism
$L_{\INT_p}^\vee\otimes W(A)\rightarrow \V\otimes W(A)$ mapping
$s$ to $\sdr$, such that $\widetilde{t}\otimes W(A)/(p)$ maps
$L^1\otimes (W(A)/(p))$ to $\V^1\otimes (W(A)/(p))$. Then $t$ is
described as follows. Let $\mathbf{G}$ be $\mathbf{Aut}(\V,\sdr)$.
Then $\V^1\otimes (W(A)/(p))\subseteq \V\otimes (W(A)/(p))$ is
induced by a cocharacter of $\mathbf{G}\otimes (W(A)/(p))$ which
lifts to a cocharacter of $\mathbf{G}\otimes W(A)$. So we have a
splitting $\V_{W(A)}=\widehat{\V^1}\oplus \widehat{\V^0}$ such
that $\widehat{\V^1}\otimes (W(A)/(p))=\V^1\otimes (W(A)/(p))$.
The isomorphism $\widetilde{t}$ is then of the form $\bigl(
\begin{smallmatrix} a & b \\ c & d
\end{smallmatrix} \bigr)$, with $a\in
\mathrm{Isom}(L^1_{W(A)},\widehat{\V^1})$, $b\in
\mathrm{Hom}(L^0_{W(A)},\widehat{\V^1})$, $c\in p\cdot
\mathrm{Hom}(L^1_{W(A)},\widehat{\V^0})=(p)\otimes
\mathrm{Hom}(L^1_{W(A)},\widehat{\V^0})$, and $d\in
\mathrm{Isom}(L^0_{W(A)},\widehat{\V^0})$. Now $t$ is given by
base-change to $W_m(A)$ of the corresponding elements in Hom or
Isom, denoted by $\bigl(
\begin{smallmatrix} a_m & b_m \\ c_m & d_m
\end{smallmatrix} \bigr)$. We remark that $c_m$ is of the form $p\otimes e_m$, with $e_m$
an element in $(\mathrm{Hom}(L^1_{W(A)},\widehat{\V^0}))\otimes
W_m(A)$. The element $t$ is independent of choices of splittings,
although we need a splitting to write down the matrix.

The morphism $I'_m\rightarrow K\I'_m$ is then given by mapping
$\bigl(
\begin{smallmatrix} a_m & b_m \\ c_m & d_m
\end{smallmatrix} \bigr)$ to $\bigl(
\begin{smallmatrix} a_m & b_m \\ p \cdot e_m & d_m
\end{smallmatrix} \bigr)$. Here $p \cdot e_m\in (\mathrm{Hom}(L^1_{W(A)},\widehat{\V^0}))\otimes
I_{m+1}(A)$ is induced by the multiplication by $p$ map
$W_m(A)\rightarrow I_{m+1}(A)$.

Now we explain why $K\I'_m$ and $K\zeta_m$ descent to $\ES_0$. Let
$U^{ij}$ be $U^i\cap U^j$ and $R^{ij}$ be its affine coordinate
ring. Let $K\I'^{i}_m$ be the restriction of $K\I'_m$ to $U^i_0$,
and $K\zeta^i_m: K\I'^{i}_m\rightarrow G_m$ be the restriction of
$K\zeta_m$ to $U^i_0$. There is a canonical isomorphism
$c^{ij}:\V_{U^i}|_{W(R^{ij}_0)}\rightarrow
\V_{U^j}|_{W(R^{ij}_0)}$ given by the composition of canonical
isomorphisms $\V_{U^i}|_{W(R^{ij}_0)}\rightarrow
\mathbb{D}(\A_{\ES_0})(W(R^{ij}_0))\rightarrow
\V_{U^j}|_{W(R^{ij}_0)}$. The isomorphism $c^{ij}$ respects $\sdr$
and the $F-V$ module structures. The universal element
$K\I'^{i}_m|_{U^{ij}_0}\stackrel{\mathrm{id}}{\rightarrow}K\I'^{i}_m|_{U^{ij}_0}$
is such that it lifts to a section of $\I'_m$ over an affine
regular fppf cover (namely, $\I'^i_m|_{U^{ij}_0}$), so to define a
morphism, it suffices to work with $\Spec A$-points with $\Spec
A/U^{ij}_0$ regular such that there is a regular fppf affine cover
$\Spec A'\rightarrow \Spec A$ with a commutative diagram
$$\xymatrix{\Spec A'\ar[r]^{t'}\ar[d]&\I'^i_m|_{U^{ij}_0}\ar[d]^{\pi}\\
\Spec A\ar[r]^{t}&K\I'^i_m|_{U^{ij}_0}.}$$

Viewed as a $A'$-point, $t=\pi\circ t'$ is induced by a
$W(A')$-point $\widetilde{t'}$ of $\I'$ described as before. Let
$A''=A'\otimes_A A'$ and $i_1:A'\rightarrow A''$,
$i_2:A'\rightarrow A''$ be the two embeddings, then
$i_1^*t=i_2^*t$. Composing $c^{ij}$ induces a well defined
$B$-point of $K\I'^{j}_m|_{U^{ij}_0}$, denoted by $c^{ij}\circ t$.
Noting that $i_1^*(c^{ij}\circ t)=i_1^*c^{ij}\circ
i_1^*t=i_2^*c^{ij}\circ i_2^*t=i_1^*(c^{ij}\circ t)$, $c^{ij}\circ
t$ descents to an $A$-point of $K\I'^{j}_m|_{U^{ij}_0}$ by
faithfully flat descent. It is independent of the choices of $A'$,
as one could work with the universal morphism
$\I'^i_m|_{U^{ij}_0}\rightarrow K\I'^i_m|_{U^{ij}_0}$ and then
base-change to $A$. So $c^{ij}$ induces an isomorphism
$K\I'^{i}_m|_{U^{ij}_0}\rightarrow K\I'^{j}_m|_{U^{ij}_0}$,
denoted by $c^{ij}$, which is clearly $KG'_m$-equivariant. But the
$c^{ij}$s satisfy the cocycle condition as they are canonical, so
one glues $K\I'^{i}_m$s to a $KG'_m$-torsor over $\ES_0$. By
construction of $\zeta_m$ in \ref{construction of zeta} which uses
only the $F-V$ module structure on
$\mathbb{D}(\A_{\ES_0})(W(R^{ij}_0))$, we see that
$K\zeta^i_m|_{U^{ij}_0}=K\zeta^j_m|_{U^{ij}_0}\circ c^{ij}.$ So
the $K\zeta^i_m$s also glue.

One has to fix a finite open affine cover $\{U^i\}_{1\leq i\leq
r}$ of $\ES$ as well as a collection of homomorphisms
$f^i:R^i=O_{U^i}\rightarrow W(R^i_0)$ to construct $K\I'_m$ as
well as $K\zeta_m$. It turns out that the glued scheme and
morphism are unique.

\begin{lemma}
The glued torsor and equivariant morphism are independent of
choices of $\{(U^i,f^i)\}_i$s.
\end{lemma}
\begin{proof}
It suffices to show that for any collection
$\{U^{i_k},g^{i_k}\}_{i,k}$, the glued torsor and equivariant
morphism are the same as those given by $\{(U^i,f^i)\}_i$. Here
$\{U^{i_k}\}_k$ is a finite open affine cover of $U^i$, and
$g^{i_k}$ is a homomorphism $R^{i_k}=O_{U^{i_k}}\rightarrow
W(R^{i_k}_0)$. Note that we do not assume any compatibility in
$f^i$ and $g^{i_k}$.

For $U^{i_k}\subseteq U^i$ and $U^{j_l}\subseteq U^j$, we write
$R^{i_kj_l}$ for the coordinate ring of $U^{i_k}\cap U^{j_l}$. We
write $\V_{U^i}|^{f^i}_{W(R^{i_kj_l})}$ for
$\V_{U^i}\otimes_{W(R^i),f^i}{W(R^{i_kj_l})}$, here we write $f^i$
for the composition $R^i\stackrel{f^i}{\rightarrow}
W(R^i_0)\rightarrow W(R^{i_kj_l}_0)$ by abusing notations. We have
canonical isomorphisms
$$\xymatrix{&\V_{U^i}|^{f^i}_{W(R^{i_kj_l})}\ar[d]^{d^{i}}\\
\V_{U^{i_k}}|^{g^{i_k}}_{W(R^{i_kj_l})}\ar[r]^(0.45){d^{i_k}}&\mathbb{D}(\A_{\ES_0}(W(R_0^{i_kj_l})))\ar[d]^{(d^{j})^{-1}}\ar[r]^(0.55){(d^{j_l})^{-1}}&\V_{U^{j_l}}|^{g^{j_l}}_{W(R^{i_kj_l})}\\
&\V_{U^i}|^{f^i}_{W(R^{i_kj_l})} }$$ The composition of the
horizontal morphisms (resp. vertical morphisms) is $c^{i_kj_l}$
(resp. $c^{ij}$). So we have a commutative diagram
$$\xymatrix@C=0.3cm{K\I'^{i}_m|_{U^{i_kj_l}_0}\ar[rr]^{c^{ij}}\ar[dd]_{(d^{i_k})^{-1}\circ d^i}\ar[dr]&& K\I'^{j}_m|_{U^{i_kj_l}_0}\ar[dd]^{(d^{j_l})^{-1}\circ d^j}\ar[dl]\\
&G_m\\
K\I'^{i_k}_m|_{U^{i_kj_l}_0}\ar[rr]^{c^{i_kj_l}}\ar[ur]&&K\I'^{j_l}_m|_{U^{i_kj_l}_0}\ar[ul],}$$which
means that the glued torsors and morphisms are unique.
\end{proof}

The glued torsor and morphism over $\ES_0$ will still be denoted
by $K\I'_m$ and $K\zeta_m$ respectively. The induced morphism
$\ES_0\rightarrow [KG'_m\backslash KG_m]$ will be denoted by
$K\zeta_{m,\#}$. Fibers of $K\zeta_{m,\#}$ are level $m$ strata.
\begin{corollary}
The morphism $K\zeta_{m,\#}$ is smooth, and it is determined by
the Shimura datum.
\end{corollary}
\begin{proof}
The smoothness follows from Theorem \ref{smoothness} and that the
natural morphism $\I'_m\rightarrow K\I'_m\times_{\ES_0} U_0$ is
faithfully flat. The morphism $\zeta_{m,\#}$ is uniquely
determined by the Shimura datum together with a collection
$\{(U^i,f^i)\}_i$, and so is $K\zeta_{m,\#}$. But by the previous
lemma, $K\zeta_{m,\#}$ is independent of choices of
$K\zeta_{m,\#}$s, so it is uniquely determined by the Shimura
datum.
\end{proof}

\subsection[More properties about level $m$ stratifications]{More properties about level $m$ stratifications}
We will start with an analog of \cite{foliation-Oort} Theorem 1.3 which
plays an essential role there. \begin{proposition}\label{constancy
theorem} Let $\ES_0^s$ be a level $m$ stratum, $x\in \ES_0^s$ be a
point. Then there is a quasi-finite fppf cover $T=\Spec A$ of $\ES_0^s$,
such that there exists an isomorphism $(\A_x[p^m])_T\rightarrow
\A_T[p^m]$ whose induced isomorphism on Dieudonn\'{e} modules
(truncated displays) respects the Hodge-Tate tensor.
\end{proposition}
\begin{proof}
We work with $U_0$ and consider the commutative diagram $$\xymatrix{
\I'_m\ar[r]\ar[d]^{\zeta_m} &U_0\ar[d]^{\zeta_{m,\#}}\ar[r]&\mathscr{B}\mathscr{T}_m\otimes \kappa\ar[dr]\ar[r]^{\Phi_m}&\mathscr{D}isp_m\otimes \kappa\ar@{=}[d]\\
G_m\ar[r]& [G'_m\backslash G_m]\ar[r]&[KG'_m\backslash
G_m]\ar[r]&[K\GL'_m\backslash \GL_m]\otimes \kappa.}$$ Let $G_m^s$ be the
$G'_{m,\overline{\kappa}}$-orbit in $G_{m,\overline{\kappa}}$
corresponding to $U_0^s$. Let $g\in G_m(\overline\kappa)$ be the
image under $\zeta_m$ of a section of $\I'_{m, x}$. By \cite{EGA4} Corollary 17.16.3, there is a quasi-finite \'{e}tale cover $T=\Spec A$ of $U_0^s$, such that $\I'_{m}(T)\neq\emptyset$. The composition $\mathfrak{g}:T\rightarrow \I'_{m}\rightarrow G_{m}$
factors through $G_m^s$.

We claim that there is a finite flat
cover $T'$ of $T$, such that there exists $h\in G'_m(T')$ with $\mathfrak{g}_{T'}=h\cdot g_{T'}$. It suffices to show that, over a finite flat cover of $G_m^s$, the orbit morphism
$G'_{m,\overline{\kappa}}\rightarrow G_m^s$, $h\rightarrow h\cdot
g$ admits a section. Let $H_g$ be the stabilizer of $g$ in
$G'_{m,\overline{\kappa}}$, and $(H_g)_{\mathrm{red}}^0$ be the
reduced identity component. Then
$G'_{m,\overline{\kappa}}/(H_g)_{\mathrm{red}}^0$ is a finite flat
cover of $G_m^s$. We only need to show that
$G'_{m,\overline{\kappa}}\rightarrow
G'_{m,\overline{\kappa}}/(H_g)_{\mathrm{red}}^0$ admits a section.

Let $g_0$ be the image of $g$ in $G_0:=G_1$, and $H_{g_0}$ be the stabilizer of $g_0$ in $G'_0:=G'_1$. Then by \cite{zipaddi} Theorem 8.1, $(H_{g_0})^0_{\mathrm{red}}$ is unipotent. Let $H_{g_0}^m$ be the inverse image of $(H_{g_0})^0_{\mathrm{red}}$ in $G'_{m,\overline{\kappa}}$, then $H_{g_0}^m$ is unipotent and contains $(H_g)_{\mathrm{red}}^0$. This implies that $(H_g)_{\mathrm{red}}^0$ is unipotent, and hence $G'_{m,\overline{\kappa}}\rightarrow
G'_{m,\overline{\kappa}}/(H_g)_{\mathrm{red}}^0$ admits a section.

We still write Let $h$ for its image in $KG'_m(T')$ and $K\GL'_m(T')$. Then $h$
gives an isomorphism between truncated displays attached to
$(\A_x[m])_{T'}$ and $(\A_{U_0^s}[m])_{T'}$. But by Theorem \ref{Lau's
main results} (3), there is a finite flat cover $T''$ of $T'$, such
that $h_{T''}$ comes from an isomorphism of BT-$m$s.
\end{proof}

Fiber of $\ES_0\rightarrow[KG'_m\backslash G_m] \rightarrow [K\GL'_m\backslash \GL_m]\otimes \kappa$ are precisely classical level $m$ strata (see 1.3). For $s\in |[K\GL'_m\backslash \GL_m]|$, we write $\ES^{\mathrm{cL},s}_0$ for the corresponding stratum. It is a locally closed subscheme of $\ES_{0,\overline{\kappa}}$, and, as a set\footnote{One can replace ``set'' by ``scheme'', see Remark \ref{union as schemes}.}, it is a union of level~$m$ strata.
\begin{lemma}\label{constancy theorem without add str}
There is a quasi-finite fppf cover $T=\Spec A$ of $\ES^{\mathrm{cL},s}_0$ such that $\A_T[p^m]$ is constant.
\end{lemma}
\begin{proof}
The proof is the same as the previous proposition, but we should consider the following commutative diagram:
$$\xymatrix@C=0.75cm{
\I'_m\times^{G'_m}\GL'_{m,\kappa}\ar[r]\ar[d] &U_0\ar[d]\ar[r]&\mathscr{B}\mathscr{T}_m\otimes \kappa\ar[dr]\ar[r]^{\Phi_m}&\mathscr{D}isp_m\otimes \kappa\ar@{=}[d]\\
\GL_{m,\kappa}\ar[r]& [\GL'_m\backslash \GL_m]\otimes \kappa\ar[rr]& &[K\GL'_m\backslash \GL_m]\otimes \kappa.}$$
\end{proof}
\begin{proposition}\label{open-closed level m}
Each $\ES^{s'}_0$ is a union of connected components of $\ES^{\mathrm{cL},s}_0$. In particular, it is open and closed in $\ES^{\mathrm{cL},s}_0$.
\end{proposition}
\begin{proof}
The proof is similar to \cite{foliation-Oort} Theorem 3.3, but we prefer to write down the details. The proposition follows from the following statement.

($*$) Let $y$ be the generic point of an irreducible component of $\ES^{\mathrm{cL},s}_0$, then $\overline{\{y\}}\subseteq \ES^{s'}_0$ if $y\in \ES^{s'}_0$.

This is because any irreducible component intersecting $\overline{\{y\}}$ lies in $\ES^{s'}_0$, while in one connected component of $\ES^{\mathrm{cL},s}_0$, one can joint two irreducible components which do not intersect by other irreducible components.

Now we prove ($*$). We fix an $x\in \ES_0^{s'}(\overline{\kappa})$. By passing to an affine cover, we can assume that $T:=\overline{\{y\}}=\Spec(A)$ with $A$ an integral domain of finite type over $\overline{\kappa}$. Let $\AAut_{\overline{\kappa}}(\A_x[p^m])$ be the sheaf of automorphism of $\A_x[p^m]$. By \cite{foliation-Oort} Lemma 2.4, it is represented by a $\overline{\kappa}$-scheme of finite type. Let $\overline{K}$ be an algebraically closed field containing $A$ and all residue fields $k(x)$, $x\in \AAut_{\overline{\kappa}}(\A_x[p^m])$. The assumption of ($*$) means that there is an isomorphism $f:\A_x[p^m]_{\overline{K}}\rightarrow \A_T[p^m]\otimes_A\overline{K}$, whose induced isomorphism on truncated displays respects the Hodge-Tate tensors. By Lemma \ref{constancy theorem without add str}, there is a quasi-finite and surjective morphism $T'=\Spec(B)\rightarrow T$, such that $\A_T[p^m]\times_T T'$. By taking the reduced induced structure and passing to affine covers of irreducible components, we can assume that $B$ is an integral domain of finite type over $\overline{\kappa}$.

Let's write $-\otimes -$ for tensor product over $\overline{\kappa}$. We have $A\hookrightarrow B\hookrightarrow \overline{K}$, and that $\overline{K}\otimes B$ is an integral domain with fraction field $L$. The isomorphism $f_L$ respects the Hodge-Tate tensors. By constancy of $\A_{T'}[p^m]$, we also have an isomorphism $\alpha:\A_T[p^m]\times_A(B\otimes \overline{K})\rightarrow \A_x[p^m]\otimes B\otimes \overline{K}$. The composition $\alpha_L\circ f_L$ is an automorphism of $\A_x[p^m]\otimes L$, so it is given by a $\overline{K}$-point $g$ of $\AAut_{\overline{\kappa}}(\A_x[p^m])$. Now the composition $\alpha^{-1}\circ g_{B\otimes \overline{K}}$ is such that its induced isomorphism on truncated displays respects the Hodge-Tate tensors after tensoring $W_m(L)$. But then it has to respect the tensors, as $W_m(B\otimes \overline{K})\subseteq W_m(L)$. In particular, ($*$) holds.
\end{proof}
\begin{remark}\label{union as schemes}
The previous proposition implies that $\ES^{\mathrm{cL},s}_0$ is a finite disjoint union of level $m$ strata as schemes, and hence smooth.
\end{remark}

\section[Geometry of Newton strata]{Geometry of Newton strata}

Let $T=\Spec(R)$ be an affine scheme over $\ES$ such that $R$ is $p$-adically complete, $p$-torsion free and equipped with a lift of Frobenius $\delta$. Let $M$ be $\mathbb{D}(\A)(R)$ and $\varphi$ be the Frobenius. If $\I_+(T)\neq\emptyset$, then $\varphi$ determines a $\intG(R)$-$\delta$-conjugacy class of $\intG(R)\mu^{\sigma}(p)\delta(\intG(R))$. More precisely, for $t\in \I_+(R)$ and $M^i:=t(L^i_R)$, $i=1,2$, we have a commutative diagram $$\xymatrix@C=0.45cm{
L_R\ar[r]^(0.3){\xi}\ar[drrr]_{\mu^{\sigma}(p)}&\delta^*(L^1_R\oplus L^0_R)\ar[rr]^(0.6){p\oplus \mathrm{id}}&&\delta^*L_R\ar@<-1ex>[d]_{\xi^{-1}}\ar[r]^(0.35){\delta^*t}&\delta^*(M^1\oplus M^0)\ar[rrr]^(0.6){(\frac{\varphi}{p})^\lin\oplus \varphi^\lin}&&&M\ar[r]^{t^{-1}}&L_R.\\
&&&L_R\ar@<-1ex>[u]_{\xi}\ar[urrrrr]_{g_t}}$$
Noting that the composition of horizontal morphisms is $\varphi^\lin$ and that $g_t\in \intG(R)$, we have $\varphi^\lin=g_t\mu^{\sigma}(p)$. So $\varphi^\lin$ gives $\{hg_t\mu^{\sigma}(p)\delta(h)^{-1}\mid h\in \intG(R)\}$ by considering all elements in $\I(R)$.

Let $F$ be the fraction field of $W:=W(\overline{\kappa})$. Let $C(\intG)$ (resp. $B(G)$) be the set of $\intG(W)$-$\sigma$-conjugacy (resp. $G(F)$-$\sigma$-conjugacy) classes in $G(F)$. Let $C(\intG,\mu^\sigma)$ be the set of $\intG(W)$-$\sigma$-conjugacy classes in $\intG(W)\mu^{\sigma}(p)\intG(W)$, and $B(\intG,\mu^\sigma)$ be the image of $C(\intG,\mu^\sigma)\hookrightarrow C(\intG)\twoheadrightarrow B(G)$. For a $\overline{\kappa}$-point $x$ of~$\ES_0$, it lifts to a $W$-point $\widetilde{x}$ of $\ES$, and $\I_{+,\widetilde{s}}(W)\neq\emptyset$. Applying the above construction to $R=W$, we get an element in $C(\intG,\mu^\sigma)$. This induces maps $\cent:\ES_0(\overline{\kappa})\rightarrow C(\intG,\mu^\sigma)$ and $\newt:\ES_0(\overline{\kappa})\rightarrow B(\intG,\mu^\sigma)$. The pre-image under $\cent$ (resp. $\newt$) of an element in $C(\intG,\mu^\sigma)$ (resp. $B(\intG,\mu^\sigma)$) is the set of $\overline{\kappa}$-points of a central leaf (resp. Newton stratum).

\subsection[Group theoretic preparations]{Group theoretic preparations}
Let $B\subseteq\intG$ be a Borel subgroup and $T\subseteq B$ be a maximal torus. Then $T$ splits over a unramified extension of $\mathbb{Q}_p$. Let $\Gamma$ be $\mathrm{Gal}(\overline{\kappa}/\mathbb{F}_p)$, then $X^*(T)$ (resp. $X_*(T)$), the group of characters (resp. cocharacters) is a $\Gamma$-module. Let $\pi_1(G)$ be the quotient of $X_*(T)$ by the coroot lattice. Let $W_G$ be the Weyl group of $G$, $\widetilde{W}_G:=\mathrm{Norm}_G(T)(F)/T(W)\cong W_G\ltimes X_*(T)$ be the extended affine Weyl group and $W_a$ be the affine Weyl group. We have a canonical exact sequence
$$\xymatrix{0\ar[r]&W_a\ar[r]&\widetilde{W}_G\ar[r]&W_G\ar[r]&0}.$$Let $\Omega\subseteq \widetilde{W}_G$ be the stabilizer of the alcove corresponding to the Iwahoric subgroup of $G(F)$ given by the preimage of $B(\overline{\kappa})$. We define the length function on
$\widetilde{W}_G$ by
\begin{subeqn}\label{length func}
l(wr)=l(w), \text{ for }w\in W_a, r\in \pi_1(G).
\end{subeqn}
\subsubsection{Results about $B(G)$}To a $G(F)$-$\sigma$-conjugacy class $[b]$, Kottwitz defines two functorial invariants $\nu_G(b)\in (X_*(T)_{\mathbb{Q}}/W_G)^\Gamma\cong X_*(T)_{\mathbb{Q},\mathrm{dom}}^\Gamma$ and $\kappa_G(b)\in \pi_1(G)_\Gamma$ in \cite{isocys with addi}. These two invariants determines $[b]$ uniquely.

We consider the partial order $\leq$ on $X_*(T)_{\mathbb{Q}}$ given by $\chi'\leq\chi$ if and only if $\chi-\chi'$ is a linear combination of positive roots with positive rational coefficients. We write $\overline{\sigmu}$ the average of its $\Gamma$-orbit.
\begin{proposition}
For $b\in \intG(W)\sigmu(p)\intG(W)$, we have $\nu_G(b)\leq \overline{\sigmu}$ and $\kappa_G(b)=\sigmu_*$. Here $\sigmu_*$ is the image of $\sigmu$ in $\pi_1(G)_\Gamma$.
\end{proposition}
\begin{proof}
This is \cite{class of F-isocrys} Theorem 4.2.
\end{proof}
By works of Gashi, Kottwitz, Lucarelli, Rapoport and Richartz, we have $$B(\intG,\sigmu)=\{[b]\in B(G)\mid \nu_G(b)\leq \overline{\sigmu}\text{ and }\kappa_G(b)=\sigmu_*\}.\ \ \ \text{(See \cite{VW} 8.6).}$$

To each $G(K)$-$\sigma$-conjugacy class $[b]$, one defines $M_b$ to be the centralizer in $G$ of $\nu_G(b)$, and $J_b$ be the group scheme representing $$J_b(R)=\{g\in G(R\otimes_{\mathbb{Q}_p}K)\mid gb=b\sigma(g)\}.$$
The group $J_b$ is a inner form of $M_b$ which does not depend on the choices of representatives in $[b]$ (see \cite{isocys with addi} 5.2).
\begin{definition}
For $[b]\in B(G)$, the defect of $[b]$ is defined by $$\mathrm{def}_G(b)=\mathrm{rank}_{\mathbb{Q}_p}G-\mathrm{rank}_{\mathbb{Q}_p}J_b.$$
\end{definition}
Hamacher gives a formula for $\mathrm{def}_G(b)$ using root data.
\begin{proposition}
Let $w_1,\cdots,w_l$ be the sums over all elements in a Galois orbit of absolute fundamental weights of $G$. For $[b]\in B(G)$, we have $$\mathrm{def}_G(b)=2\cdot\sum_{i=1}^l\{\langle \nu_G(b),w_i\rangle\},$$
where $\{\cdot\}$ means the fractional part of a rational number.
\end{proposition}
\begin{proof}
This is \cite{geo of newt PEL} Proposition 3.8.
\end{proof}

\subsubsection{Minimal points and fundamental elements}Let us write $K$ for $G_W(W)$, $K_1$ for $\mathrm{ker}(K\rightarrow G_W(\overline{\kappa}))$ and $\mathcal{I}$ for the Iwahoric subgroup attached to $B$.
\begin{definition}
A element $x\in G(F)$ is called minimal if for any $y\in K_1gK_1$, there is a $g\in K$ such that $y=gx\sigma(g)^{-1}$.
\end{definition}
By \cite{VW} Remark 9.1, if $x$ is minimal, then any element in the $K\text{-}\sigma$-orbit of $x$ is again minimal.
\begin{definition}
An element $x\in \widetilde{W}$ is fundamental if $\mathcal{I}x\mathcal{I}$ lies in a single $\mathcal{I}\text{-}\sigma$-orbit.
\end{definition}
If $x$ is a representative of a fundamental element, then it is minimal. By \cite{Nie Sian} Theorem 1.3, it is also straight.
\begin{theorem}
Each orbit in $B(\intG, \mu^{\sigma})$ contains a fundamental element in $W_G\mu^{\sigma}(p)W_G$.
\end{theorem}
\begin{proof}
This is \cite{Nie Sian} Proposition 1.5.
\end{proof}
\subsubsection{Truncations of level 1 and Ekedahl-Oort strata}
For a dominant cocharacter $\chi\in X_*(T)$, we write $W_\chi$ for the Weyl group of the centralizer of $\chi$, and ${}^\chi W$ for the set of elements $w$ which are the
shortest representatives of their cosets $W_\chi w$. Let $x_\chi=w_0w_{0,\chi}$ where $w_0$ denotes the longest element of $W_G$ and
where $w_{0,\chi}$ is the longest element of $W_\chi$. Let $\tau_\chi=x_\chi\chi(p)$. Then $\tau_\chi$ is the shortest element of $W_G\chi(p)W_G$.
\begin{theorem}
Let $\mathcal{T}=\{(w,\chi)\in W_G\times X_*(T)_{\mathrm{dom}}\mid w\in {}^\chi W\}$. Then the map assigning
to $(w, \chi)$ the $K\text{-}\sigma$-conjugacy class of $K_1w\tau_\chi K_1$ is a bijection between $\mathcal{T}$ and the set of
$K\text{-}\sigma$-conjugacy classes in $K_1\backslash G(F)/K_1$.
\end{theorem}
\begin{proof}
This is \cite{truncat level 1} Theorem 1.1 (1).
\end{proof}
\begin{remark}\label{dim of minimal E-O}
Let $(w,\mu^{\sigma})\in\mathcal{T}$ be corresponding to the $K\text{-}\sigma$-conjugacy classes in $K_1\backslash G(F)/K_1$ attached to a fundamental element in $[b]\in B(\intG,\mu^{\sigma})$. Then by the group theoretic arguments in \cite{geo of newt PEL} 7.3,  $l(w)=\langle2\rho, \nu_G(b)\rangle$, where $\rho$ is the half-sum of positive roots of $G$. This is the dimension of the Ekedahl-Oort stratum corresponding to $w$ by \cite{EOZ}.
\end{remark}

\subsection[Central leaves]{Central leaves}
By Theorem \ref{cent are big level m}, all central leaves are given by level $m$ stratum for $m$ big enough. So we can apply results about level $m$ strata to central leaves.
\begin{corollary}
Each central leaf is a smooth, equi-dimensional locally closed subscheme of $\ES_{0,\overline{\kappa}}$. It is open and closed in the classical central leaf containing it, and closed in the Newton stratum containing it.
\end{corollary}
\begin{proof}
We only need to explain why it is closed in the Newton stratum, as other statements following from Theorem \ref{main theorem level m--locally closed+sm} and Proposition \ref{open-closed level m}. But by \cite{foliation-Oort} Theorem 2.2, a classical central leaf is closed in the classical Newton stratum, so a central leaf is closed in the classical Newton stratum. In particular, it is closed in the Newton stratum.
\end{proof}
Let $x,y\in \ES_{0}(\overline{\kappa})$ be two points, and $X,Y$ be their attached $p$-divisible groups. For any $n\in \INT_{>0}$, we write $X_n$ (resp. $Y_n$) for its $p^n$-kernel. Let $I^{X,Y}$ (resp. $I_n^{X,Y}$) be the subset of elements $f\in \Isom(X,Y)$ (resp. $f\in \Isom(X_n,Y_n)$) such that $\Dieu(f):\Dieu(Y)\rightarrow \Dieu(X)$ (resp. $\Dieu(f):\Dieu(Y_n)\rightarrow \Dieu(X_n)$) respects the Hodge-Tate tensors. For $N\in \INT$ such that $N\geq n$, we write $\Phi_n:\Isom(X,Y)\rightarrow \Isom(X_n,Y_n)$ resp. $\Phi_n^N:\Isom(X_N,Y_N)\rightarrow \Isom(X_n,Y_n)$ for the natural maps induced by restricting to $p^n$-kernels.
\begin{lemma}
For each $n\in \INT$, there is an integer $N(X,Y,n)$, such that for any $N\geq N(X,Y,n)$, we have $\Phi_n(I^{X,Y})=\Phi_n^N(I^{X,Y}_N)$.
\end{lemma}
\begin{proof}
This is essentially the proof of \cite{foliation-Oort} Lemma 1.5. More precisely, by the proof there, $\Phi_n(\Isom(X,Y))$ is finite, and there is $N_1$ such that $\forall\ N\geq N_1$, we have $\Phi_n(\Isom(X,Y))=\Phi_n^N(\Isom(X_N,Y_N)).$ In particular, $\forall\ N\geq N_1$, $\Phi_n(I^{X,Y})$, $\Phi_n^N(I^{X,Y}_N)$ and $\Phi_n^N(\Isom(X_N,Y_N))$ are finite. But $$\Phi_n(I^{X,Y})\subseteq \Phi_n^{N+1}(I^{X,Y}_{N+1})\subseteq \Phi_n^N(I^{X,Y}_N)\text{ and }\Phi_n(I^{X,Y})=\bigcap_N\Phi_n^N(I^{X,Y}_N),$$ so there is $N_0\geq N_1$ such that $\Phi_n(I^{X,Y})=\Phi_n^{N_0}(I^{X,Y}_{N_0})$ which proves the claim.
\end{proof}
\begin{remark}
In\cite{foliation-Oort}, Lemma 1.5. shows that there is $N(X,Y,n)$, such that for any $N\geq N(X,Y,n)$, $\Phi_n(\Hom(X,Y))=\Phi_n^{N}(\Hom(X_N,Y_N))$. One could ask for similar results for $\Hom^0$, homomorphisms respecting Hodge-Tate tensors. However, the phase \emph{homomorphisms from} $X_n$ \emph{to} $Y_n$ \emph{respecting Hodge-Tate tensors} does not quite make sense. As $\Dieu(X_n)^\otimes$ is constructed from $\Dieu(X_n)$ by taking tensors, duals, and symmetric/exterior powers, and there are no obvious maps between $\Dieu(X_n)^\otimes$ and $\Dieu(Y_n)^\otimes$ induced by $f\in \Hom(X_n,Y_n)$ unless $f$ is an isomorphism.
\end{remark}
For $x$ and $X$ as above, we denote by $G^X$ (resp. $G_n^X$) the group of elements in $\Aut(X)$ (resp. $\Aut(X_n)$) whose induced map on Dieudonn\'{e} module respects Hodge-Tate tensors. If $y\in \ES_0(\overline{\kappa})$ is in the central leaf crossing $x$, then $I^{X,Y}$ (resp. $I_n^{X,Y}$) is a trivial $G^X$-torsor (resp. $G_n^X$-torsor).
\begin{corollary}\label{boundedness in a central leaf}
Let $C\subseteq\ES_{0,\overline{\kappa}}$ be a central leaf. For each $n\in \INT_{>0}$, there is an integer $N(n,C)$, such that for any $N\geq N(n,C)$, we have $\Phi_n^N(I^{X,Y}_N)=\Phi_n(I^{X,Y})$ for any $x,y\in C(\overline{\kappa})$.
\end{corollary}
\begin{proof}
Fix $x\in C(\overline{\kappa})$ with attached $p$-divisible group $X$, by the previous lemma, for any $n\in \INT_{>0}$, there is an integer $N(X,n)$, such that for $N\geq N(X,n)$, we have $\Phi_n^N(G^X_N)=\Phi_n(G^X)$. Then for any $y\in C(\overline{\kappa})$ with attached $p$-divisible group $Y$, we have $\Phi_n^N(I^{X,Y}_N)=\Phi_n(I^{X,Y})$ and $\Phi_n^N(G^Y_N)=\Phi_n(G^Y)$, as $I^{X,Y}$ (resp. $I_N^{X,Y}$) is a trivial $G^X$-torsor (resp. $G_N^X$-torsor). One takes $N(n,C)=N(X,n)$ and finishes the proof.
\end{proof}
\begin{proposition}\label{dim of cent}
Let $[b]\in B(\intG,\sigmu)$ be a conjugacy class, and $\ES_0^b$ be the corresponding Newton stratum. Then any central leaf in $\ES_0^b$ is of dimension $\langle2\rho,\nu_G(b)\rangle$.
\end{proposition}
\begin{proof}
We only need to prove that central leaves in $\ES_0^b$ are of the same dimension. The dimension formula then follows from Remark \ref{dim of minimal E-O}. Let $x\in \ES_0^b(\overline{\kappa})$ be a point, and $\ES_0^x$ be the connected component of the central leaf containing $x$. The $G'_{\overline{\kappa}}$-torsor $\I'_x$ is trivial, we will fix such a section as well as its lifting $t_x$ in $\I'(W)$\footnote{One could also do this by choosing a lifting $\widetilde{x}$ of $x$, and considering a section of $\I_{+,\widetilde{x}}$.}. The image of $t_x$ in $\I(W)$ will be denoted by the same notation. By \cite{LRKisin} Proposition 1.4.4, an element $g\in G(K)$ such that $g^{-1}b\sigma(g)\in \intG(W)\mu^\sigma(p)\intG(W)$ gives a point $gx\in \ES_0^b(\overline{\kappa})$, denoted by $y$, by considering the the Dieudonn\'{e} submodule $t_x(gV_W)$. Multiplying $g$ by some $p^r$ if necessary, we can assume that there is an isogeny $f:\A_x\rightarrow \A_y$ of degree $p^{n_1}$ whose induced morphism on Dieudonn\'{e} modules respects the Hodge-Tate tensors. We write $\ES_0^y$ for the connected of the central leaf containing $y$. Note that for each central leaf, there is a point of the form $gx$.

We will construct a connected regular affine $\overline{\kappa}$-scheme $T$, with dominant and quasi-finite morphisms $\ES_0^x\leftarrow T\rightarrow \ES_0^y$. We assume that the symplectic embedding we choose is of good reduction\footnote{This is always possible, for example, by Zarhin's trick, one can take $(V_{\INT_p}\oplus V_{\INT_p}^\vee)^4$.}. Let $\ES_0^x$ be the irreducible component of the central leaf containing (the image of)~$x$. Let $n_2$ be such that all level $n_2$ strata (in $\ES_0^x$) are central leaves, and $n$ be $n_1+n_2$. Let $N$ be $n+N(n,\ES_0^x)$ with $N(n,\ES_0^x)$ be as in the previous corollary. By Proposition \ref{constancy theorem}, there is a quasi-finite fppf cover $s:T=\Spec R\rightarrow \ES_0^x$, such that there exists an isomorphism $i:(\A_x[p^N])_T\rightarrow \A_T[p^N]$ whose induced map on Dieudonn\'{e} modules respects Hodge-Tate tensors. We write $i$ for the restriction of $i$ to $p^n$-kernels.

By passing to an irreducible component of $T$ containing a closed point $z$ mapping to $x$, we can assume that $T$ is integral, and that $\tau:T\rightarrow \ES_0^x$ is dominant, quasi-finite and of finite type. Modifying $x$ (and hence $y=gx$) by a prime to $p$ isogeny if necessary, we can assume that $z$ is in the smooth locus of $T$. We could then pass to a smooth affine neighborhood of $z$, which will still be denoted by $T$.

Let $H$ be $\mathrm{ker}(f)$ and $\EH$ be $i(H\times T)$ in $\A_T[p^n]$. We consider the abelian scheme $\A_T/\EH$ over $T$.

(1) The polarization and level structure on $\A_T$ descend to $\A_T/\EH$, which gives a quasi-finite morphism $\tau':T\rightarrow \mathscr{A}_{g,K'}$. The polarization descents as $H$ is isotropic for the form defined by the polarization, while $i$ respects these forms. The quasi-finiteness follows from the arguments in \cite{foliation-Oort} page 279, which shows that for $u\in \mathscr{A}_{g,K'}(\overline{\kappa})$, $\tau(\tau'^{-1}(u))$ is finite.

(2) For $u\in T(\overline{\kappa})$, $\tau'(u)$ factors through $\ES_0$. By our choice of $i$, the isomorphism $i_u:\A_x[p^n]\rightarrow \A_u[p^n]$ lifts to an isomorphism $i_0:\A_x[p^\infty]\rightarrow \A_u[p^\infty]$ respecting Hodge-Tate tensors. Choosing a $t_u\in\I(W)$ as at the beginning of the proof, and using the commutative diagram
$$\xymatrix{\A_x[p^\infty]\ar[r]^{i_0}\ar[d]^f&\A_u[p^\infty]\ar[d]\\
\A_y[p^\infty]\ar[r]&\A_u[p^\infty]/i_0(H),}$$
we can identify $\Dieu(\A_u[p^\infty]/i_0(H))\hookrightarrow \Dieu(\A_u[p^\infty])$ with $t_u(gV_W)$ for some $g\in G(K)$ such that $g^{-1}b\sigma(g)\in \intG(W)\mu^\sigma(p)\intG(W)$. By \cite{LRKisin} Proposition 1.4.4, this gives a point $gu\in \ES_0^b(\overline{\kappa})$.

Let $\R$ be the lifting of $R$ which is in particular regular. Let $\M$ (resp. $\M'$) be $\mathbb{D}(\A_T)(\R)$ (resp. $\mathbb{D}(\A_T/\EH)(\R)$). The isogeny $\A_T\rightarrow \A_T/\EH$ gives $\M'\hookrightarrow \M$ which induces an isomorphism $\M'[\frac{1}{p}]\rightarrow \M[\frac{1}{p}]$. One could then translate the section $\sdr\in \M^\otimes$ to a section of $\M'[\frac{1}{p}]^\otimes$, denoted by $\sdr'[\frac{1}{p}]$.

We claim the followings.

(3) $\sdr'[\frac{1}{p}]$ extends to a section $\sdr'$ over $\R$, and $\mathbb{J}:=\IIsom_{\R}\big((V_W,s)_{\R}, (\M',\sdr')\big)$ is a $G_W$-torsor.

We will show how to finish the proof with (3), and then give a proof of it.

(4) By (3), passing to a quasi-finite affine scheme which is \'{e}tale over $T$, we could assume that $\mathbb{J}(\R)$ and $\I(\R)$ are non-empty. One could identify $\M'\hookrightarrow \M$ as an element $g\in G(\R[\frac{1}{p}])$ by choosing an element in $\mathbb{J}(\R)$ resp. $\I(\R)$. One could then identify $$\xymatrix{\M'\ar[r]\ar[d]^{\varphi'}&\M\ar[d]^\varphi\\
\M'\ar[r]&\M} \text{\ \  and \ \ }\xymatrix{\M\ar[r]^{g\cdot}\ar[d]^{\varphi}&\M\ar[d]^\varphi\\
\M\ar[r]^{g\cdot}&\M.}$$

The Hodge filtration on $\M_T$ is induced by a cocharacter of $G_T$ which lifts to a cochatacter of $G_{\R}$ that is conjugate to $\mu$. Let $\M=\mathrm{Fil}^1(\M)\oplus \mathrm{Fil}^0(\M)$ be the splitting induced by this cocharacter. Here $\mathrm{Fil}^1(\M)$ is a lift of the Hodge filtration. Let $\mathrm{Fil}_0(\M)$ be the submodule of $\M$ generated by $\varphi(\mathrm{Fil}^0(\M))$. Then $\varphi$ induces a $\sigma$-isomorphism $\mathrm{Fil}^0(\M)\rightarrow \mathrm{Fil}_0(\M)$. By the above identification, we have $\M'=g\mathrm{Fil}^1(\M)\oplus g\mathrm{Fil}^0(\M)$, and $g\mathrm{Fil}_0(\M)\subseteq \M'$ a direct summand. The commutative diagram $$\xymatrix{\sigma^*(\mathrm{Fil}^0(\M))\ar[r]^{\varphi^\lin}\ar[d]^{g^\sigma}&\mathrm{Fil}_0(\M)\ar[d]^g\\
\sigma^*(g\mathrm{Fil}^0(\M))\ar[r]^{\varphi^\lin}&g\mathrm{Fil}_0(\M)}$$ is such that all but the lower horizontal map are isomorphisms, so it has to be an isomorphism. But then $g\mathrm{Fil}^1(\M)$ will be a lift of the Hodge filtration on $\M'$.

Now one verifies the conditions in \cite{LRKisin} 1.4.7, and applies \cite{LRKisin} Proposition 1.4.9. This implies that we have $\ES_0^x\stackrel{\tau}{\leftarrow}\Spec R\stackrel{\tau'}{\rightarrow}\ES_0^y$ with $\tau$ quasi-finite and dominant,  and $\tau'$ quasi-finite. So we have $\mathrm{dim}\ES_0^x\leq \mathrm{dim}\ES_0^y$. But we can get $\mathrm{dim}\ES_0^x\geq \mathrm{dim}\ES_0^y$ by exactly the same proof. So we have $\mathrm{dim}\ES_0^x= \mathrm{dim}\ES_0^y$ and that $\tau'$ is also dominant.

\emph{Proof of (3)}. By the construction of \cite{CIMK} Proposition 1.3.2, there is a line $L$ in some $\intG$-representation $\intV^{\frac{1}{2}\otimes}$ constructed from $\intV$ (by taking sums, tensor products, duels, symmetric/exterior products), such that $\intG$ is the stabilizer of $L$, and that $\intV^\otimes=\intV^{\frac{1}{2}\otimes}\otimes (\intV^{\frac{1}{2}\otimes})^\vee$. The line bundle $\mathcal{E}:=L\times^{\intG}\I_\R$ is a direct summand of $\M^{\frac{1}{2}\otimes}$, and $\mathcal{E}[1/p]\subseteq \M[1/p]^{\frac{1}{2}\otimes}\cong\M'[1/p]^{\frac{1}{2}\otimes}$ extends to a direct summand of rank one of $\M'^{\frac{1}{2}\otimes}_U$, denoted by $\mathcal{E}$. Here $U\subseteq \Spec \R$ is an open affine subscheme of form $\Spec \R_f$ with $f\notin(p)$.

(3.i) $\mathbb{J}_1:=\IIsom_{U}\big((V_W,L_W)_U, (\M',\mathcal{E}')\big)$, isomorphism mapping $L_U$ to $\mathcal{E}'$, is a $G_W$-torsor. We only need to check the faithfully flatness. It is already a $G_W$-torsor over $\R[\frac{1}{p}]$, so we only need to check this at points in the special fiber. Let $t$ be a closed point of $U_{\overline{\kappa}}$, and $s$ be a closed point of $\mathbb{J}_{1,t}$. By (2), $s$ is in the closure of $\mathbb{J}_1\times_UU[1/p]$ in $\IIsom_{U}(V_U, \M'_U)$, and $\mathbb{J}_{1,t}$ is a $G_{\kappa}$-torsor. In particular, we have $O_{\mathbb{J}_1,s}$ is flat over $W$, and
\begin{equation*}
\begin{split}
\mathrm{dim}(O_{\mathbb{J}_1,s})-\mathrm{dim}(O_{U,t})&=(\mathrm{dim}(O_{\mathbb{J}_1,s}[1/p])+1)-(\mathrm{dim}(O_{U,t}[1/p])+1)\\
&=\mathrm{dim}(G)=\mathrm{dim}(\mathbb{J}_{1,t}).
 \end{split}
 \end{equation*}
 This implies that $\mathbb{J}_1\rightarrow U$ is faithfully flat, and hence a $G_W$-torsor.

(3.ii) Noting that $\R_{(p)}\cong O_{U,(p)}$ is a DVR and $G_W$ is quasi-split, we see that $\mathbb{J}_{1,\R_{(p)}}$ is a trivial torsor by \cite{Nisnevich} Theorem 7.1. The existence of $\R_{(p)}$-sections implies that the rational map $\sdr'[\frac{1}{p}]:\R\dashrightarrow \V^\otimes$ is defined at all points of codimension 1, and hence extends to $\R$, as $\V^\otimes$ is a group and $\R$ is normal. This extended section will be denoted by $\sdr'$, and it is necessarily the closure of $\sdr'[\frac{1}{p}]$ in $\V^\otimes$.

(3.iii) By the same argument as in (3.i), $\mathbb{J}=\IIsom_{\R}\big((V_W,s)_{\R}, (\M',\sdr')\big)$ is a $G_W$-torsor.
\end{proof}

\subsubsection[Slope filtrations and completely slope divisible $p$-divisible groups]{Slope filtrations and completely slope divisible $p$-divisible groups}
We refer to \cite{family of p-div with cons NP} Definition 1.1 for the definition of slope filtrations and Definition 1.2 for that of completely slope divisible $p$-divisible groups.

Let $\G$ be a $p$-divisible group over $S$, and $\G_\bullet:0=\G_0\subseteq\G_1\subseteq\cdots\subseteq\G$ be its slope filtration. Then $\G^i:=\G_{i}/\G_{i-1}$ is again a $p$-divisible group over $S$. Assume that (contravariant) Dieudonn\'{e} modules exists for $p$-divisible groups over $S$, and write $\M$ for the Dieudonn\'{e} module of $\G$. The slope filtration induces surjections of $p$-divisible groups $\G\twoheadrightarrow \G/\G_1\twoheadrightarrow\G/\G_2\twoheadrightarrow\cdots\twoheadrightarrow\G/\G_{n-1}.$ Taking Dieudonn\'{e} modules, we get $\M^\bullet:0\subseteq \M^1\subseteq \M^2\subseteq\cdots\subseteq\M^n=\M$, with $\M^i$ the Dieudonn\'{e} module of $\G/\G_{n-i}$. Clearly, $\M_i:=\M^i/\M^{i-1}$ is the Dieudonn\'{e} module of $\G^{n-i}$. We will call $\M^\bullet$ the slope filtration of $\M$, it is such that $\M^i\subseteq \M^{i+1}$ is a locally direct summand.

For $x\in \ES_0(k)$, with $k$ algebraically closed, the slope filtration on $\A_x[p^\infty]$ induces a filtration on $M:=\mathbb{D}(\A_x[p^\infty])(W(k))$, denoted by $M^\bullet$. Fixing an isomorphism $t:\intV^\vee\otimes W(k)\rightarrow M$ respecting the Hodge-Tate tensors, we could view $G_{W(k)}$ as a subgroup of $\GL(M)$, and one can talk about $G_{W(k)}$-\emph{split filtrations} (see \cite{CIMK} 1.1.2 for the definition) of~$M$.
\begin{lemma}\label{slope fil is G-split}
The slope filtration on $M$ is $G_{W(k)}$-split.
\end{lemma}
\begin{proof}
Let $K$ be the fraction field of $W(k)$. By \cite{isocys with addi}, $M_K^\bullet$ is $G_K$-split, so by \cite{CIMK} Proposition 1.1.4, $M^\bullet$ is $G_{W(k)}$-split.
\end{proof}
\begin{lemma}
Let $\ES_0^b$ be a Newton stratum of $\ES_0$. Then there exists a point $x\in \ES_0^b(\overline{\kappa})$ such that $\A_x[p^\infty]$ is completely slope divisible.
\end{lemma}
\begin{proof}
The proof here is given by putting together \cite{isocys with addi} 4.3 and \cite{LRKisin} Proposition 1.4.4. Let $x\in \ES_0^b(\overline{\kappa})$ be a point and $M$ be the Dieudonn\'{e} module of $\A_x[p^\infty]$. Fixing an isomorphism $t:V_{W(\overline{\kappa})}\rightarrow M$ respecting the Hodge-Tate tensors, we identify the Frobenius of $M$ with an element in $g\in G(K)$ where $K=W(\overline{\kappa})[1/p]$, and hence a $\sigma\text{-}K$-space structure on $V_K$ (see \cite{isocys with addi} 3, with $L=K$). Let $n$ be such that $n\nu(g)\in \mathrm{Hom}_K(\mathbb{G}_{m,K},G_K)$, and $E\subseteq K$ be the fixed field of~$\sigma^n$.

We have two fiber functors from $\mathrm{Rep}(G)$ to the category of finite dimensional $E$-vector spaces. Namely, $w_1((V',\rho))=V'_E$, and $w_2((V',\rho))=(V'_K)^{\rho(g)}$, where $(V'_K)^{\rho(g)}$ is defined at the end of \cite{isocys with addi} 3. $\IIsom(w_1,w_2)$ is a $G_E$-torsor, so there is a finite extension $E'$ of $E$ in $K$, such that $w_1$ and $w_2$ become isomorphic over $E'$. By replacing $n$ by $n[E':E]$, we can assume that $w_1$ and $w_2$ are isomorphic over $E$. Moreover, the natural injection $(V'_K)^{\rho(g)}\rightarrow V'_K$ induces an isomorphism $(V'_K)^{\rho(g)}\otimes_E K\rightarrow V'_K$.

Let $\alpha:w_1\rightarrow w_2$ be such an isomorphism over $E$, and $\alpha_K$ be its base-change to $K$. Let $\beta:w_2\rightarrow w_1$ be the isomorphism given by $(V'_K)^{\rho(g)}\otimes_E K\rightarrow V'_K=V'_E\otimes_EK$. Then $\beta\circ \alpha_K$ gives an automorphism of $w_1$ over $K$, and hence an element $c$ in $G(K)$. Noting that $\alpha$ and $c$ determines each other uniquely, we could assume that $c$ is such that $cg\sigma(c)^{-1}\in \intG(W(\overline{\kappa}))\mu^{\sigma}(p)\intG(W(\overline{\kappa}))$, by enlarging $E$ and changing $\alpha$ if necessary. Then by \cite{LRKisin} 1.4.2 and Proposition 1.4.4, $c^{-1}\cdot V_{W(\overline{\kappa})}$ with Frobenius induced by $g$ is a Dieudonn\'{e} module, and it comes from a point $c^{-1}x\in \ES_0^b(\overline{\kappa})$. The slope filtration on $\mathbb{D}(\A_{c^{-1}x}[p^\infty])=c\cdot V_{W(\overline{\kappa})}$ is induced by the cocharacter ${}^{c^{-1}}\!n\nu(g)$. The equality 4.3.3 of \cite{isocys with addi} means that $$cg\sigma(c)^{-1}\cdot \sigma(cg\sigma(c)^{-1})\cdot \sigma^2(cg\sigma(c)^{-1})\cdots \sigma^{n-1}(cg\sigma(c)^{-1})={}^{c^{-1}}\!n\nu(g),$$ which implies that the Dieudonn\'{e} module structure on $c^{-1}\cdot V_{W(\overline{\kappa})}$ is completely slope divisible.
\end{proof}
Let $\ES_0^b\subseteq \ES_{0,\overline{\kappa}}$ be a Newton stratum, and $x\in \ES_0^b(\overline{\kappa})$ be a point such that $\A_x[p^\infty]$ is completely slope divisible. Let $C_x\subseteq \ES_0^b$ be the central leaf containing $x$. By Lemma \ref{slope fil is G-split}, $V_{W(\overline{\kappa})}^{\vee\bullet}:=t^{-1}(M^\bullet)$ is a $G_{W(\overline{\kappa})}$-split filtration of $V^\vee_{W(\overline{\kappa})}$. We write $P$ for $\mathrm{stab}_{G_{W(\overline{\kappa})}}(V_{W(\overline{\kappa})}^{\vee\bullet})$. Let $T=\Spec R$ be a $C_x$-scheme which is smooth over $\overline{\kappa}$. Let $\mathcal{R}$ be a lifting of $R$ which is smooth over $W(\overline{\kappa})$. Then by \cite{family of p-div with cons NP} Proposition 2.3, $\A_{C_x}[p^\infty]$ (and hence $\A_T[p^\infty]$) is completely slope divisible. In particular, $\A_T[p^\infty]$, and hence $\mathcal{M}:=\mathbb{D}(\A_T)(\mathcal{R})$ admits a slope filtration. We write $\mathcal{M}^\bullet$ for the slope filtration on $\mathcal{M}$.

\begin{corollary}\label{slope fil on cent gives torsor}
The scheme $I:=\IIsom_{\mathcal{R}}\big((V^\vee_{W(\overline{\kappa})}, V_{W(\overline{\kappa})}^{\vee\bullet}, s)_{\mathcal{R}},(\mathcal{M},\mathcal{M}^\bullet,\sdr)\big)$ of isomorphisms respecting the filtrations as well as the Hodge-Tate tensors is a $P$-torsor over $\mathcal{R}$.
\end{corollary}
\begin{proof}
There exists a morphism $\mathcal{R}\rightarrow \ES$ lifting $R\rightarrow C_x$, so we have
$$J:=\IIsom_{\mathcal{R}}\big((V^\vee_{W(\overline{\kappa})}, s)_{\mathcal{R}},(\mathcal{M},\sdr)\big)=\IIsom_{\mathcal{R}}\big((V^\vee_{W(\overline{\kappa})}, s)_{\mathcal{R}},(\V_{\mathcal{R}},\sdr)\big)$$is a $G_{W(\overline{\kappa})}$-torsor. By passing to a connected component of an fppf \'{e}tale affine cover, we can assume that $\mathcal{R}$ is an integral domain and that $J(\R)\neq\emptyset$. Fixing a $t\in J(\mathcal{R})$, we could view $\mathcal{M}^\bullet$ as a filtration on $V^\vee_{W(\overline{\kappa})}\otimes \mathcal{R}$. Applying Lemma \ref{slope fil is G-split} to the generic point of $T$, we see that $\M^\bullet_{\mathcal{K}}$ is induced by a cocharacter of $G$ which is geometrically conjugate to $\nu$. Here $\nu$ is a cocharacter (of $G_W$) inducing $V_{W(\overline{\kappa})}^{\vee\bullet}$.

Let $G_W\text{-}\Fil$ (resp. $\Fil$) be the $G_W$-split (resp. usual) filtrations on $V_{\R}$. By \cite{CIMK} Proposition 1.1.5, $G_W\text{-}\Fil$ is a proper smooth $\R$-scheme, and hence the morphism $G_W\text{-}\Fil\rightarrow \Fil$ is a closed immersion. The $\R$-point of $\Fil$ induced by $\M^\bullet$ is such that the generic point lies in $G_W\text{-}\Fil$, so $\M^\bullet$ is $G_W$-split.

Let $\mathcal{P}$ be $\mathrm{stab}_{G_{W(\overline{\kappa})}}(\mathcal{M}^\bullet)$. It is a parabolic subgroup of $G_{\R}$ by \cite{CIMK} Lemma 1.1.1. Then $\mathrm{trans}_{G_{W(\overline{\kappa})}}(P_\mathcal{S},\mathcal{P})\cong I_\mathcal{S}$. But $\mathrm{trans}_{G_{W(\overline{\kappa})}}(P_\mathcal{S},\mathcal{P})$ is a $P$-torsor as $P$ and $\mathcal{P}$ are of the same type (because they are of the same type at all points in $T$), so $I$ is a $P$-torsor.
\end{proof}

\subsubsection[Igusa towers]{Igusa towers}Let $x\in \ES_0^b(\overline{\kappa})$ be such that $\A_{x}[p^\infty]$ is completely slope divisible and $C_x$ be the central leaf containing $x$. Let $\G$ be $\A_{C_x}[p^\infty]$. We write $\G_{x,\bullet}$ (resp. $\G_{\bullet}$) for the slope filtration, and $\G_{x}^i$ (resp. $\G^i$) for $\G_{x,i}/\G_{x,i-1}$ (resp. $\G_{i}/\G_{i-1}$).

Let $M$ be the Dieudonn\'{e} module of $\G_x$, and $M^\bullet$ be the slope filtration. We fix, once an for all in this subsection, an isomorphism $t:V^\vee_{W(\overline{\kappa})}\rightarrow M$ respecting the Hodge-Tate tensors and hence identify $V^\vee_{W(\overline{\kappa})}$ and $M$. The canonical slope decomposition $\G_x=\oplus\G_{x}^i$ gives a decomposition of Dieudonn\'{e} modules $M=\oplus M_i$, and hence a cocharacter $\nu$ of $G_{W(\overline{\kappa})}$. We write $P$ (resp. $P'$) for the stabilizer of $M^\bullet$ in $G_{W(\overline{\kappa})}$ (resp. $\GL(M)$), $U$ (resp. $U'$) for the unipotent radical, and $L$ (resp. $L'$) for the centralizer of $\nu$.

Let $\Spec R$ be an affine scheme over $C_x$. By smoothness of $\ES_W$, there is a morphism $j:\Spec (W(R))\rightarrow \ES_W$. We write $\mathcal{W}$ for $j^*\V\cong\mathbb{D}(\G_R)(W(R))$, and $\mathcal{W}^\bullet$ for the slope filtration induced by that on $\mathbb{D}(\G_R)(W(R))$.

For $m\in \INT_+$, let $J'_m$ be the presheaf such that for a $R$-algebra $A$, $J'_m(A)$ is the set of $W_m(A)$-linear isomorphisms $M\otimes W_m(A)\rightarrow \mathcal{W}\otimes_{W(R)}W_m(A)$ mapping $M^\bullet|_{W_m(A)}$ to $\mathcal{W}^\bullet|_{W_m(A)}$. We define $J_m$ to be the sub presheaf of $J_m'$ such that the $W_m(A)$-linear isomorphisms $M\otimes W_m(A)\rightarrow \mathcal{W}\otimes_{W(R)}W_m(A)$ respect Hodge-Tate tensors.
\begin{lemma}\label{closed and indep of liftings}
\

(1) Both $J_m$ and $J'_m$ are represented by smooth affine $R$-schemes. Moreover, $J_m$ is a closed subscheme of $J'_m$.

(2) Both $J_m$ and $J'_m$ are independent of the choice of $j$.
\end{lemma}
\begin{proof}
For (1), we can assume that $j$ factors through an affine subscheme of $\ES_W$, whose $p$-adic completion is denoted by $\R$. The homomorphism $\R\rightarrow W(R)$ induced by $j$ will still be denoted by $j$. We can choose a Frobenius on $\R$ such that $j$ is compatible with Frobeni. Let $\mathcal{M}$ be $\mathbb{D}(\A[p^\infty])(\mathcal{R})$ and $\mathcal{M}^\bullet$ be the slope filtration. Let $I'$ be $\IIsom_{\R}\big((M,M^\bullet)_{\R},(\M,\M^\bullet)\big)$, and $I$ be $\IIsom_{\R}\big((M,M^\bullet,s)_{\R},(\M,\M^\bullet,\sdr)\big)$. Then $I'$ is a $P'$-torsor and by Corollary \ref{slope fil on cent gives torsor}, $I$ is a $P$-torsor. Now $J'_m$ (resp $J_m$) is given by first applying the Greenberg functor to $I'\rightarrow \Spec \R$ (resp. $I\rightarrow \Spec \R$), and then pulling back via $\Spec R\rightarrow \W_m(\R)$ given by $\R\stackrel{j}{\rightarrow} W(R)\rightarrow W_m(R)$. In particular, (1) holds.

For (2), let $j_1$ and $j_2$ be two homorphisms $\R\rightarrow W(R)$. Then the canonical isomorphism $i_{12}:j^*_1\V\rightarrow \mathbb{D}(\A_R)(W(R))\rightarrow j^*_2\V$ respects the Hodge-Tate tensors as well as the slope filtrations. This gives an isomorphism between the torsors induced by $j_1$ and $j_2$. It also satisfies the cocycle condition.
\end{proof}
\begin{definition}
The Igusa tower $J_{b,m}$ is the sheaf over $C_x$ which attaches to a $C_x$-scheme $T$ the set of isomorphisms $t:\oplus_i\G_x^i[p^m]\rightarrow \oplus_i\G^i_T[p^m]$, such that

(1) $t$ extends \'{e}tale locally to any $m'\geq m$;

(2) for any affine scheme $\Spec R$ over $T$, the element in $J'_m/U'_m(R)$ induced by $t_R$ lies in $J_m/U_m(R)$.
\end{definition}
Let $\Gamma_b$ be $L(W)\cap \mathrm{Aut}(\G_x)$ (here the intersection is via the Dieudonn\'{e} functor which is an equivalence), and $\Gamma_{b,m}$ is the quotient of $\Gamma_b$ by the subgroup which is identity on $\G_x[p^m]$. By Corollary \ref{boundedness in a central leaf}, there is in integer $N$, such that for any $ y\in C_x(\overline{\kappa})$ and any $m\geq N$, $J_{b,m,y}(\overline{\kappa})$ is a $\Gamma_{b,m}$-torsor. Here we use the canonical identifications $\G_x=\oplus_i\G_x^i[p^m]$ and $\G_y=\oplus_i\G_y^i[p^m]$.

\begin{proposition}
$J_{b,m}$ is representable. Moreover, for $m\geq N$, $J_{b,m}\rightarrow C_x$ is finite etal\'{e} and Galois, with Galois group $\Gamma_{b,m}$.
\end{proposition}
\begin{proof}
As in \cite{coho of PEL} section 4, the sheaf $J'_{b,m}/C_x$ such that $J'_{b,m}(T)$ is the set of isomorphisms $t:\oplus_i\G_x^i[p^m]\rightarrow \oplus_i\G^i_T[p^m]$ which extends \'{e}tale locally to any $m'\geq m$ is representable. For any open affine subscheme $T=\Spec R$ of $J'_{b,m}$, the pull back to $T$ of the universal isomorphism on $J'_{b,m}$ gives a morphism $T\rightarrow J'_m/U'_m$. By Lemma \ref{closed and indep of liftings} (1), to make it factor through $J_m/U_m$ gives a closed condition, and by Lemma \ref{closed and indep of liftings} (2), the closed subschemes obtained by using different $T$s glue to a closed sunscheme of $J'_{b,m}$. This is precisely $J_{b,m}$.

For the rest statements, as in \cite{coho of PEL} Proposition 4, it suffices to prove that for each closed point $y\in C_x$, there is an isomorphism $t:\oplus_i\G_x^i\otimes_{\overline{\kappa}}R\rightarrow \oplus_i\G^i\otimes_{C}R$ whose induced map $\R\rightarrow I'/U'$ factors through $I/U$. Here we write $C$ for $C_x$, and $R$ for $O_{C,y}^\wedge$ for simplicity, and $\R$, $I$ and $I'$ are as in the previous lemma. By \cite{unitary shv} Lemma 3.4, there is an isomorphism $t_0:\oplus_i\G_x^i\otimes_{\overline{\kappa}}R\rightarrow \oplus_i\G^i\otimes_{C}R$.

Let $K$ be a algebraically closed field containing $R$. Then there is an isomorphism $t_1:\G_x\otimes_{\overline{\kappa}}K\rightarrow \G_R\otimes_RK$ respecting both the slope filtrations and the Hodge-Tate tensors. Its induced map $\oplus_i\G_x^i\otimes_{\overline{\kappa}}K\rightarrow \oplus_i\G^i_R\otimes_RK$ will still be denoted by $t_1$. Then $g:=t^{-1}_1\circ t_{0,K}$ is a automorphism of $\oplus_i\G_{x,K}^i$, and hence it is defined over~$\overline{\kappa}$ and still denoted by $g$. Now $t_0\circ g^{-1}_R$ is precisely what we need.
\end{proof}
\textbf{From now on, we always assume }$m\geq N$ when working with $J_{b,m}$s. For $m'\geq m$, there is a natural projection $q:J_{b,m'}\rightarrow J_{b,m}$ induced by restricting to the $p^m$-torsion. This morphism is finite \'{e}tale. Let $J_{b}=\varprojlim_m J_{b,m'}$, it is equipped with the action of $\Gamma_b$. Let $T_b$ be the group of self quasi-isogenies of $\G_x$ respecting the the tensors. We will show, as in \cite{coho of PEL}, that the $\Gamma_b$-action on $J_b$ extends to a sub-monoid $S_b\subseteq T_b$.

For $\rho\in T_b$, we white $\rho=\oplus_i\rho^i$ for the decomposition to isoclinic factors. If $\rho^{-1}$ is an isogeny, we write $e_i(\rho)\geq f_i(\rho)$ respectively for the minimal and maximal integer such that $\mathrm{ker}(p^{f_i})\subseteq \mathrm{ker}({\rho^i}^{-1})\subseteq\mathrm{ker}(p^{e_i}).$ We define $$S_b=\{\rho\in T_b|\rho^{-1}\text{ is a isogeny}, f_{i-1}(\rho)\geq e_i(\rho), \forall\ i\geq2\}.$$It is not hard to see $S_b$ is a monoid, and that $p^{-1}$ and $fr^{-B}:=\oplus_ip^{-\lambda_iB}$ are in $S_b$. Moreover, the proof of \cite{unitary shv} Lemma 2.11 also works here, and hence $T_b=\langle S_b, p, fr^B\rangle$.
\begin{proposition}\label{rou action}
Let $m$, $\rho$ be as before, and $e=e_1(\rho)$. There is a unique finite flat group scheme $\mathcal {H}\subseteq \G_{J_{b,m}}[p^e]$, such that the corresponding subgroups in $\G^i_{J_{b,m}}$ are $t(\mathrm{ker}({\rho^{i}}^{-1}))$. The abelian scheme $\A/\mathcal{H}$, together with the polarization and level structure, induces a morphism $\rho_*:J_{b,m}\rightarrow \mathscr{A}$ which factors through $C_x$. Moreover, it induces a morphism $\rho:J_{b,m}\rightarrow J_{b,m-e}$ of Igusa towers.
\end{proposition}
\begin{proof}
The existence of $\mathcal{H}$ is proved in \cite{unitary shv} Lemma 3.6. For $z\in J_{b,m}(\overline{\kappa})$, it is a pair $(y,j)$ with $y\in C_x(\overline{\kappa})$ and $j:\G_x[p^m]\rightarrow \G_y[p^m]$ an isomorphism whose induced map on Dieudonn\'{e} modules respects the Hodge-Tate tensors. Here we use the canonical identifications $\G_x=\oplus_i\G_x^i[p^m]$ and $\G_y=\oplus_i\G_y^i[p^m]$. By our assumption, $j$ lifts to an isomorphism $j':\G_x\rightarrow \G_y$ whose induced map on Dieudonn\'{e} modules respects Hodge-Tate tensors. The isogeny $\rho^{-1}:\G_x\rightarrow \G_x$ gives a element $g\in L(K)$ such that $g^{-1}b\sigma(g)\in G_W(W)\sigmu G_W(W)$. It gives, via $j'$, a point $gy\in \ES_0(\overline{\kappa})$ by \cite{LRKisin}. The isomorphism $\G_x\rightarrow \G_{gy}$ which is the push out of $j'$ via $\rho^{-1}$ is an isomorphism respecting Hodge-Tate tensors. So $gy\in C_x(\overline{\kappa})$. By the proof of Proposition \ref{dim of cent}, for each open affine subscheme $\Spec(R)\subseteq J_{b,m}$, the composition $\Spec(R)\rightarrow J_{b,m}\stackrel{\rho_*}{\rightarrow} \mathscr{A}$ factors through $\ES_0$. This factorization is necessarily unique by \cite{LRKisin} Proposition 1.4.9, and hence  glue to a morphism $i:J_{b,m}\rightarrow \ES_0$ which necessarily factors through $C_x$.

Let $H$ be $\mathrm{ker}(\rho^{-1})$. Using the identification $\G_{J_{b,m}}/\mathcal{H}\cong i^*\G$, the isomorphism $$\oplus_i\G_x^i[p^{m-e}]\stackrel{\rho^{-1}}{\rightarrow} \oplus_i(\G_x^i/H^i)[p^{m-e}]\stackrel{t}{\rightarrow}\oplus_i(\G^i/\mathcal{H}^i)[p^{m-e}]\cong\oplus_i\G^i[p^{m-e}]$$
gives the morphism $J_{b,m}\rightarrow J_{b,m-e}$.
\end{proof}
\begin{remark}\label{Frob on Igusa}
As in \cite{coho of PEL}, by the same proof above, we can also define $\sigma$-semi-liner action of $F:\G_x\rightarrow \G_x^{(p)}$ on Igusa towers. More precisely, $Frob:J_{b,m}\rightarrow J_{b,m-1}$ is induced by the abelian scheme $\A^{(p)}=\A/\A[F]$. Here by $\sigma$-semi-liner, we mean the following diagram
$$\xymatrix{J_{b,m}\ar[r]^{Frob}\ar[d] &J_{b,m-1}\ar[d]\\
C_x\ar[r]^{\sigma}&C_x}$$ is commutative. Moreover, $Frob=q\circ \sigma_{J_{b,m}}$. Here we write $\sigma$ for the absolute Frobenius.
\end{remark}
\subsection[Foliations]{Foliations}
\subsubsection[Rapoport-Zink formal schemes of Hodge type]{Rapoport-Zink formal schemes of Hodge type}\label{R-Z for Hodge}
Rapoport-Zink formal schemes of Hodge type are first defined and constructed by Wansu Kim in \cite{RZ for Hodge}. Howard and Pappas give in \cite{RZ for spin} a more direct construction relying on the existence of the integral model. We will follow \cite{RZ for spin} in this paper.

Let's fix some notations as in \cite{RZ for spin} 2.1.1. We write $\Nil_W$ for the category of $W$-schemes $S$
such that $p$ is Zariski locally nilpotent in $O_S$. We write $\ANil_W\subseteq\Nil_W^\mathrm{op}$ for the full subcategory of Noetherian $W$-algebras in which $p$ is nilpotent, and  $\ANil^f_W$
for the category of Noetherian adic W-algebras in which $p$ is
(topologically) nilpotent, and embed
$\ANil_W\subseteq\ANil^f_W$ as a full subcategory by endowing any W-algebra in $\ANil_W$ with its $p$-adic
topology.
We say that an adic $W$-algebra $A$ is \emph{formally finitely generated} if $A$ is
Noetherian, and if $A/I$ is a finitely generated $W$-algebra for some ideal of
definition $I\subseteq A$. Thus $\Spf(A)$ is a formal scheme which is formally of finite
type over $\Spf(W)$. If, in addition, $p$ is nilpotent in $A$, then $A$ is a quotient
of $W/(p^n)[[x_1,\cdots, x_r]][y_1,\cdots, y_s]$ for some $n, r$, and $s$.
We will denote by
$\ANil^\fsm_W\subseteq\ANil^f_W$
the full subcategory whose objects are $W$-algebras that are formally finitely
generated and formally smooth over $W/(p^n)$, for some $n\geq1$.

We start with classical Rapoport-Zink
spaces. Let $X_0$ be a $p$-divisible group over $k=\overline{k}$. The Rapoport-Zink
space $\RRZ(X_0)$ of deformations of $X_0$ up to quasi-isogeny is the functor assigning to each
scheme $S$ in $\Nil_W$ the set of isomorphism classes of pairs $(X, \rho)$, where $X$ is a $p$-divisible group over $S$, and $\rho:X_0\times_k\overline{S}\dashrightarrow X\times_S\overline{S}$ is a quasi-isogeny. Here $\overline{S}:= S\times_Wk$. As in \cite{period of p-div}, $\RRZ(X_0)$ is represented by a formal scheme $\RZ(X_0)$ over $\Spf(W)$ that is formally smooth
and locally formally of finite type over $W$. If $(X_0,\lambda_0)$ is a principal polarized $p$-divisible group, one can also define $\RRZ(X_0,\lambda_0)$ (see \cite{RZ for spin} 2.3.1 or \cite{period of p-div}). It is represented by a closed sub formal scheme of $\RZ(X_0)$, denoted by $\RZ(X_0,\lambda_0)$, which is again formally smooth
and locally formally of finite type over $W$.

Now we consider the morphism $\ES_W\rightarrow \mathscr{A}_W$ as before, which is induced by an embedding of Shimura data $(G,X)\rightarrow (\GSp(V,\psi),X')$ which are both of good reduction at $p$. For $x\in\ES_W(W)$, we write $X_0$ for $\A_x[p^\infty]$. The $G'_W$-torsor $\I'_x$ is trivial, and its sections lifts to elements in $\I'(W)$. We fix such a lifting as before, and use it to translate the Dieudonn\'{e} module structure on $\mathbb{D}(X_0)(W)$ to $V_W$. Let $V_W=\mathrm{F}^1\oplus \mathrm{F}^0$ be the splitting on $V_W$ induced by $\mu$, then under the above identification, $\mathrm{F}^1$ gives the Hodge filtration on $V_W$. Moreover, as at the beginning of this section, the Frobenius of $X_0$ gives an element $b\in B(G_W,\sigmu)$.
\begin{definition}
We define the functor $\RRZ_{G_W}^\nil:\ANil_W\rightarrow \mathrm{Sets}$ as follows. For any $R\in \ANil_W$, $\RRZ_{G_W}^\nil(R)$ is the set of isomorphism classes of triples $(X,\rho,t)$, with $(X,\rho)\in \RZ(X_0)(R)$, and $t:\mathbf{1}\rightarrow \mathbb{D}(X)^\otimes$ be such that $t[1/p]$ is Frobenius equivariant, satisfying the following conditions.

(1) For some nilpotent ideal $J\subseteq R$ with $p\in J$, the pull-back of $t$ to
$\Spec(R/J)$ is identified with $s$ under the isomorphism of isocrystals
$$\mathbb{D}(\rho):\mathbb{D}(X_{R/J})^\otimes[1/p]
\rightarrow \mathbb{D}(X_0\times_{\overline{\kappa}}R/J)^\otimes[1/p]$$
induced by the quasi-isogeny $\rho$.

(2) $\IIsom\big((\mathbb{D}(X),t),(V_W, s)_R\big)$, the sheaf of isomorphisms
of crystals on $\Spec(R)$ respecting the tensors, is a crystal of $G_W$-torsors over the (big fppf) crystalline
site $\mathrm{CRIS}(\Spec(R)/W)$.

(3) $\IIsom_R\big((\mathbb{D}(X)(R),\Fil^1, t),(V_W, \mathrm{F}^1, s)_R\big)$ is a $P_+$-torsor. Here $\Fil^1\subseteq\mathbb{D}(X)(R)$ is the Hodge filtration.

An isomorphism $(X,\rho,t)\rightarrow (X',\rho',t')$ is an isomorphism of $p$-divisible groups $X\rightarrow X'$ compatible with additional structures in the obvious way.
\end{definition}
\begin{definition}
The functor $\RRZ^\fsm_{G_W}$ on $\ANil^\fsm_W$ is defined by setting
$$\RRZ^\fsm_{G_W}(A)=\varprojlim_n\RRZ^\nil_{G_W}(A/I^n),$$
where $I$ is an ideal of definition of $A$.
\end{definition}
We have the following results by \cite{RZ for spin} or \cite{RZ for Hodge} (see \cite{RZ for spin} Theorem 3.2.1).
\begin{theorem}
The functor $\RRZ^\fsm_{G_W}$ is represented by a closed sub formal scheme $\RZ_{G_W}(X_0)$ of $\RZ(X_0)$ which is formally smooth and formally locally of finite type over $\Spf(W)$
\end{theorem}

Now we briefly discuss the construction of $\RZ_{G_W}(X_0)$ in \cite{RZ for spin}. Let $i:\ES_W^\wedge\rightarrow \mathscr{A}_W^\wedge$ be the (finite) morphism of smooth formal schemes obtained by $p$-adic completions. Let $\pi:X\rightarrow \mathscr{A}_W^\wedge$ be the formal $p$-divisible obtained from $\A$. We use the same notation for the pull back to $\ES_W^\wedge$ of $\pi$. By \cite{CIMK}, we have an $O_{\ES_W^\wedge}$-morphism of crystals $s_\mathrm{cris}:\mathbf{1}\rightarrow \mathbb{D}(X)^\otimes$ which gives $s$ when pulling back to $x$. Moreover, $s_\mathrm{cris}[1/p]$ is Frobenius equivariant.

Let $\RZ(X_0,\lambda_0)$ be the symplectic Rapoport-Zink space attached to $x$, viewed as a point of $\mathscr{A}_{W}$, and $\Theta_G:\RZ_{G_W}^{\diamond}(X_0)\rightarrow\ES_W^\wedge$ be the pull back via $i$ of the natural map $\Theta:\RZ(X_0,\lambda_0)\rightarrow \mathscr{A}_{W}^\wedge$. Then $i':\RZ_{G_W}^{\diamond}(X_0)\rightarrow \RZ(X_0,\lambda_0)$ is a finite morphism of formal schemes and $\RZ_{G_W}^{\diamond}(X_0)$ is formally smooth and locally formally of finite type (\cite{RZ for spin} Proposition 3.2.5). Let $\RZ_{G_W}(X_0)$ be presheaf on $\Nil_W$ which attaches to $S\in \Nil_W$ the set of elements $(X,\rho)\in\RZ_{G_W}^{\diamond}(X_0)(S)$ such that for any field extension $k'/k$ and any $y\in S(k')$, the isomorphism $$\mathbb{D}(\rho):\mathbb{D}(X_y)^{\otimes}(W')[1/p]\rightarrow V_{W'}^{\otimes}[1/p]$$maps $y^*(s_{\mathrm{cris}})(W')$ to $s$. Here $W'$ is the Cohen ring of $k'$. By \cite{RZ for spin} Proposition 3.2.7 $\RZ_{G_W}(X_0)$ is represented by smooth formal scheme which is open and closed in $\RZ_{G_W}^{\diamond}(X_0)$, and by \cite{RZ for spin} Proposition 3.2.9, $\RZ_{G_W}(X_0)$ represents $\RRZ^\fsm_{G_W}$. The composition $\RZ_{G_W}(X_0)\hookrightarrow \RZ_{G_W}^{\diamond}(X_0)\stackrel{i'}{\rightarrow }\RZ(X_0,\lambda_0)$ is actually a closed immersion (\cite{RZ for spin} Proposition 3.2.11).
\subsubsection[The almost product structure of Newton strata]{The almost product structure of Newton strata}Now we assume, in addition, that $x\in \ES_0(\overline{\kappa})$ is such that $X_0$ is completely slope divisible. Let $\G/C_x$ be as as in the previous subsection and $\G_\bullet$ be its the slope filtration. Then by \cite{coho of PEL} Lemma 8, for any pair of integers $d\geq0$ and $r\geq d/\delta f$, there exists a canonical isomorphism $\alpha:\G^{(q^r)}[p^d]\cong \oplus_i\G^{i,(q^r)}[p^d]$. Here $q=|\kappa|=p^f$, and $\delta=\mathrm{min}_{i=1,\cdots,k-1}(\lambda_i-\lambda_{i+1})$ with $\lambda_i$ the slope of $\G^i$. As in \cite{coho of PEL} Lemma 8, $\alpha$ is compatible with additional structures, but in the following sense.
\begin{lemma}\label{split of slope fil by fro}
Let $\Spec R/C_x$ be an affine scheme and $\Spec W_d(R)\rightarrow \ES_W$ be a lifting. Let $\mathcal{N}$ be $\mathbb{D}(\G^{(q^r)})(W(R))$, and $\mathcal{N}^\bullet$ be the slope filtration. We write $-_d$ for the reduction to $W_d$. Let $M=\oplus_iM_i$ be as at the beginning of the previous subsection, and $I=\IIsom_{W_d(R)}((M,M^\bullet,s),(\mathcal{N}, \mathcal{N}^\bullet,\sdr))$. Then $\alpha$ induces a section of the natural projection $I\rightarrow I/U$ via the identification $M=\oplus_iM_i$.
\end{lemma}
\begin{proof}
It suffices to assume that $\Spec R$ is open affine in $C_x$. By passing to an \'{e}tale algebra, we can assume that the torsor constructed in Corollary \ref{slope fil on cent gives torsor} is trivial. A section $t$ there gives a $t_d\in I(W_d(R))$. To prove the lemma, we only need to check that the composition
$$\xymatrix{M_d= \oplus_iM_{i,d}\ar[r]^(0.6){t_d} &\oplus_i\mathcal{N}_{i,d}\ar[r]^{\alpha} &\mathcal{N}_d\ar[r]^{t_d^{-1}} &M_d}$$ lies in $G_W(W_d(R))$, where $\mathcal{N}_{i}=\mathbb{D}(\G^{i,(q^r)})(W(R))$, and $t_d:\oplus_iM_{i,d}\rightarrow\oplus_i\mathcal{N}_{i,d}$ is the composition $\oplus_iM_{i,d}=\oplus_iM_d^i/M_d^{i-1}\rightarrow \oplus_i\mathcal{N}_d^i/\mathcal{N}_d^{i-1}=\oplus_i\mathcal{N}_{i,d}$.

By the construction of $\alpha$ as in \cite{unitary shv} Lemma 4.1, we see that the composition above is base-change to $W_d(R)$ of $$\xymatrix{M= \oplus_iM_{i}\ar[r]^(0.6){t} &\oplus_i\mathcal{N}_{i}\ar[r]^{\alpha'} &\mathcal{N}\ar[r]^{t^{-1}} &M}$$ which is of the form $p^\nu\varphi$ for some pro-cocharacter $\nu$ of $G_W$. The construction implies that it is an isomorphism and hence respecs Hodge-Tate tensors.
\end{proof}
Let $\RZ_{G_W}^{n,d}(X_0)\subseteq\RZ_{G_W}(X_0)$ be the closed sub formal scheme classifying quasi-isogenies $\rho:X_0\times \overline{S}\dashrightarrow X\times_S\overline{S}$ such that $p^n\rho$ and $p^{d-n}\rho^{-1}$ are both isogenies (or equivalently, $p^n\rho$ is an isogeny whose kernel is killed by $p^d$). It is the pull back via $\RZ_{G_W}(X_0)\hookrightarrow \RZ(X_0)$ of $\RZ^{n,d}(X_0)$. The Let $\RZ_{G_W}^{n,d}(X_0)$s for a direct system as $n,d$ vary, and the direct limit is $\RZ_{G_W}(X_0)$. For $\rho\in T_b$, action by $\rho$ induces a morphism $\RZ_{G_W}^{n,d}(X_0)\rightarrow \RZ_{G_W}^{n+n(\rho),d+d(\rho)}(X_0)$, with $n(\rho)$ resp. $d(\rho)$ the smallest integer such that $p^{n(\rho)}\rho$ and $p^{d(\rho)-n(\rho)}\rho$ are isogenies. Similarly, the $\sigma$-semi-linear action of $F:X_0\rightarrow X_0^{(p)}$, $Frob:\RZ_{G_W}(X_0)\rightarrow \RZ_{G_W}(X_0)$, given by $(H,\rho)\rightarrow (H,\rho\circ F^{-1})$, induces a semi-linear morphism $\RZ_{G_W}^{n,d}(X_0)\rightarrow \RZ_{G_W}^{n+1,d+1}(X_0)$

Let $\RZ_{G_W}^{n,d,-}(X_0)$ be the reduced fiber of $\RZ_{G_W}^{n,d,-}(X_0)$. To simplify notations, we write $\bM$ for $\RZ_{G_W}(X_0)$ and $\overline{\bM}^{n,d}$ for $\RZ_{G_W}^{n,d,-}(X_0)$. For $m\geq d$, we define a morphism $\pi:J_{b,m}\times \overline{\bM}^{n,d}\rightarrow \mathscr{A}_{\overline{\kappa}}$ as follows. Let $t:\oplus_iX_0^i[p^m]\rightarrow \oplus_i\G_{J_{b,m}}^i[p^m]$ be the universal isomorphism. By the previous lemma, the isomorphism $$X_0^{(q^r)}[p^m]=\oplus_iX_0^{i,(q^r)}[p^m]\rightarrow \oplus_i{\G^{i,(q^r)}_{J_{b,m}}}[p^m]\stackrel{\alpha^{-1}}{\rightarrow} \G_{J_{b,m}}^{(q^r)}[p^m]$$ is compatible with additional structures. We still write $t$ for its restriction to $p^d$-kernels. Let $(X,\rho,s)$ be the universal object on $\overline{\bM}^{n,d}$, then $p^n\rho$ is an isogeny with $\mathrm{ker}(p^n\rho)\subseteq X_0[p^d]$. So $\mathrm{ker}(p^n\rho^{(p^r)})\subseteq X_0^{(p^r)}[p^d]$. The polarized abelian scheme $\mathrm{p}_1^*\A/\mathrm{p}_1^*t(\mathrm{p}_2^*\mathrm{ker}(p^n\rho^{(p^r)}))$ (with  level structure) gives $\pi$. Here $\A$ is the abelian scheme on $J_{b,m}$, and $\mathrm{p}_i$ is the projection of $J_{b,m}\times \overline{\bM}^{n,d}$ to the $i$-th factor.

\begin{lemma}\label{finiteness of almost prod map}
The morphism $\pi$ is finite.
\end{lemma}
\begin{proof}
It follows essentially from the Siegel cases. More precisely, let $C_x^{\GSp}$ be the central leaf crossing (the image of ) $x$ in $\mathscr{A}_{\overline{\kappa}}$, and $\overline{b}$ be the Newton polygon of $b$. Let $J_{\overline{b},m}^{\GSp}$ be the Igusa cover of $C_x^{\GSp}$. We have a commutative diagram
$$\xymatrix{J_{b,m}\ar[r]\ar[d]&J_{\overline{b},m}^{\GSp}\ar[d]\\
C_x\ar[r]&C_x^{\GSp}}$$ induced by the universal isomorphism on $J_{b,m}$. We also have a closed embedding $\overline{\bM}^{n,d}\hookrightarrow \overline{\bM}^{n,d}_{\GSp}$ such that the universal quasi-isogeny on $\overline{\bM}^{n,d}$ is the pull back of the one on $\overline{\bM}^{n,d}_{\GSp}$. Here $\overline{\bM}^{n,d}_{\GSp}$ is $\RRZ(X_0,\lambda_0)^{n,d}$. By doing the construction of $\pi$ to $J_{\overline{b},m}^{\GSp}\times\overline{\bM}^{n,d}_{\GSp}$, we get $\pi':J_{\overline{b},m}^{\GSp}\times\overline{\bM}^{n,d}_{\GSp}\rightarrow \mathscr{A}^{\overline{b}}$, such that the composition
$$\xymatrix{J_{b,m}\times\overline{\bM}^{n,d}\ar[r]^i&J_{\overline{b},m}^{\GSp}\times\overline{\bM}^{n,d}_{\GSp}\ar[r]&\mathscr{A}^{\overline{b}}}$$
is $\pi$. But $\pi'$ is finite By \cite{coho of PEL} Proposition 10, and $i$ is finite by construction, so $\pi$ is finite.
\end{proof}

Noting that $\overline{\bM}^{n,d}$ could be singular, but each of its connected components has an open dense smooth locus, we have the following ``weak foliation''.

\begin{proposition}
Let $U$ be the smooth locus of an irreducible component of $\overline{\bM}^{n,d}$. Then the morphism $\pi|_{J_{b,m}\times U}$ factors through $\ES_0^b$.
\end{proposition}
\begin{proof}
As in the first half of the proof of Proposition \ref{rou action}, each $\overline{\kappa}$-point of $J_{b,m}\times U$ factors through $\ES_0$. But by the proof of Proposition \ref{dim of cent}, the statement follows from \cite{LRKisin} Proposition 1.4.9.
\end{proof}
The above result is weak, but it is enough to compute the dimension of Newton strata.
\begin{corollary}
The Newton stratum $\ES_0^b$ is of dimension $\langle\rho,\mu+\nu(b)\rangle-\frac{1}{2}\mathrm{def}(b).$
\end{corollary}
\begin{proof}
Notations as above, we have $\pi|_{J_{b,m}\times U}$ factors through $\ES_0^b$, and it is quasi-finite by Lemma \ref{finiteness of almost prod map}. So we have $\mathrm{dim}(\ES_0^b)\geq\mathrm{dim}(J_{b,m})+\mathrm{dim}(\overline{\bM}^{n,d})$, for all $m,d,n,r$. But $\pi$ induces a finite surjection $J_{b,m}(\overline{\kappa})\times \overline{\bM}^{n,d}(\overline{\kappa})\rightarrow \ES_0^b(\overline{\kappa})$ when restricting to $\overline{\kappa}$-points for $d$ big enough. So we have $\mathrm{dim}(\ES_0^b)=\mathrm{dim}(C_x)+\mathrm{dim}(\overline{\bM})$. By \cite{zhu-aff gras in mixed char} Corollary 3.13,  $\mathrm{dim}(\overline{\bM})=\langle\rho,\mu-\nu(b)\rangle-\frac{1}{2}\mathrm{def}(b)$, and by  Theorem \ref{dim of cent}, $\mathrm{dim}(C_x)=\langle2\rho,\nu(b)\rangle$. One deduces the formula immediately.
\end{proof}
If one is willing to use a bigger $r$, one has the following ``strong foliation''
\begin{proposition}
For $r$ big enough, the morphism $\pi$ factors through $\ES_0^b$, and induces a finite morphism to it. Moreover, notations as before, we have the followings.

(1) $\pi=(fr^f\times 1)\circ \pi$;

(2) $\pi\circ q= \pi$, with $q:J_{b,m}\rightarrow J_{b,m'}$ the natural projection;

(3) $\pi\circ i= \pi$, with $i:\overline{\bM}^{n',d'}\rightarrow \overline{\bM}^{n,d}$ the natural immersion;

(4) $\pi\circ (\rho\times\rho)= \pi$, for $\rho\in S_b$, $m\geq d+d(\rho)+e(\rho)$, and $r\geq(d+d(\rho))/\delta f$;

(5) $\pi\circ (Forb^f\times Frob^f)=(1\times \sigma^f) \pi$, for $m\geq d+1$, and $r\geq(d+1)/\delta f$.
\end{proposition}
\begin{proof}
As in the proof of \cite{coho of PEL} Proposition 9, (2) and (3) are direct consequences of the construction. Let $U=\Spf(A,I)$ be an open affine connected sub formal scheme of $\bM_0:=\bM\times \Spec\overline{\kappa}$, with $I=\sqrt{I}$. By \cite{RZ for spin} Remark 2.3.5 (c), the universal family on $\Spf(A,I)$ gives family on $\Spec A$. Let $\rho$ be the quasi-isogeny on $\Spec A$, and $n'\geq n,d'\geq d$ be such that $p^{n'}\rho$ is an isogeny whose kernel is killed by $p^{d'}$. Then for $r\geq d'/\delta f$ and $m\geq d'$, $\pi|_{J_{b,m}\times \Spec A}$ factors through $\ES_0$.

Now for $\overline{\bM}^{n,d}$, we can take finitely many $U_i$s as above such that $U_i\cap\overline{\bM}^{n,d}$ form an open cover of $\overline{\bM}^{n,d}$. For $m,n',d',r$ with $m\geq d'$, $d'\geq\mathrm{max}\{d_i\}_i$, $n'\geq\mathrm{max}\{n_i\}_i$ and $r\geq d'/\delta f$ such that $p^{n'}\rho$ is an isogeny whose kernel is killed by $p^{d'}$, each attached morphism $\pi_i:J_{b,m}\times \Spec A_i\rightarrow \mathscr{A}_0$ factors through $\ES_{0, \overline{\kappa}}$ uniquely. In particular, they induce a morphism $\pi':J_{b,m}\times (\cup_iU_i)\rightarrow \ES_{0, \overline{\kappa}}$. By (3), we have $\pi=\pi'|_{J_{b,m}\times (\cup_i(\overline{\bM}^{n,d}\cap U_i))}$, here $\pi$ is the morphism attached to $m,n,d,r$, with $m,r$ as at the beginning of this paragraph. The finiteness follows from Lemma \ref{finiteness of almost prod map}. The equalities (1), (4) and (5) follow as in \cite{coho of PEL}.
\end{proof}

\

\

\end{document}